\documentclass[12pt]{amsart}

\usepackage{newtxtext}
\usepackage{newtxmath}
\usepackage{amsmath,amsthm,amsfonts,amscd}

\usepackage[a4paper, total={7in, 9 in}]{geometry}
\usepackage{epic, eepic}
\usepackage{graphicx}
\usepackage{multicol}
\usepackage{wrapfig}
\usepackage{textpos}
\usepackage{graphics}
\usepackage{color}
\usepackage{xcolor}
\usepackage{epsfig}
\usepackage{float}
\usepackage{multirow}

\usepackage{array}

\usepackage{tabularx}

\usepackage{makecell}
\usepackage{threeparttable}
\usepackage{ragged2e}

\usepackage{caption}

\usepackage{mathtools}

\usepackage{hyperref}

\newcounter{newcounter}[section]
\numberwithin{equation}{section}
\numberwithin{newcounter}{section}
\numberwithin{figure}{section}
\numberwithin{footnote}{section}

\newtheorem{theo}[newcounter]{Theorem}
\newtheorem{defi}[newcounter]{Definition}
\newtheorem{pro}[newcounter]{Proposition}
\newtheorem{lem}[newcounter]{Lemma}
\newtheorem{corollary}[newcounter]{Corollary}
\newtheorem{remark}[newcounter]{Remark}
\newtheorem{example}[newcounter]{Example}

\newcommand{\R}{{\mathbb R}}
\newcommand{\Z}{{\mathbb Z}}

\newcommand{\C}{{\mathcal{C}}}

\newcommand{\N}{{\mathcal{N}}}

\newcommand{\s}{{\mathbb{S}^1}}

\newcommand{\mz}{\textrm{mod}\,\Z}

\def\skipp#1{}

\usepackage[backend=biber, style=numeric]{biblatex}
\defbibheading{bibliography}[\bibname]{%
  \vspace*{1pt}%
  \begin{center}%
  {\scshape #1}%
  \end{center}%
  \vspace*{4pt}%
  \addcontentsline{toc}{section}{#1}%
}

\begin{document}
%\bibliographystyle{plain}

%\address{IMPA}
%\email{demelo@impa.br}
%author{Edson Vargas}
%\address{University of  S\~{a}o Paulo, Brazil}
%\email{vargas@ime.usp.br}

\title{Blaschke-Type Models for Multimodal Circle Maps}

\author{Edson de Faria}
\address{Edson de Faria - Instituto de Matem\'atica e Estat\'istica, Universidade de S\~ao Paulo}
\email{edson@ime.usp.br}

\author{Welington de Melo$^\dagger$}
\thanks{$\dagger$Welington de Melo passed away in 2016, before the completion of this paper. The present work grew out of ideas developed in collaboration with him, and the final manuscript was completed by the surviving authors.}
\address{Welington de Melo - Instituto de Matem\'atica Pura e Aplicada}
\email{dac@impa.br}

\author{Pedro A. S. Salom\~ao}
\address{Pedro A. S. Salom\~ao - Shenzhen International Center for Mathematics, SUSTech}
\email{psalomao@sustech.edu.cn}

\author{Edson Vargas}
\address{Edson Vargas - Instituto de Matem\'atica e Estat\'istica, Universidade de S\~ao Paulo}
\email{vargas@ime.usp.br}

\dedicatory{Dedicated to the memory of Welington de Melo (1946--2016)}

%\begingroup
%\renewcommand\thefootnote{$\dagger$}
%\footnotemark
%\footnotetext{Welington de Melo passed away in 2016, before the completion of this paper. The present work grew out of ideas developed in collaboration with him, and the final manuscript was completed by the surviving authors.}
%\endgroup

\begin{abstract}
For each integer $m \geq 1$, we construct a finite-dimensional family of rational maps, given by Blaschke-type products, whose restriction to the unit circle consists of $2m$-multimodal maps. We show that every post-critically finite $2m$-multimodal circle map satisfying natural dynamical conditions is topologically conjugate to a map in this family.  Moreover, we prove that this realization is unique up to rotation: two maps in the family that are topologically conjugate on the circle differ by a rigid rotation. In particular, the family provides a canonical model realizing all post-critically finite combinatorics in this class. The proofs combine a detailed description of the
critical geometry of these Blaschke-type maps with a Thurston-type fixed point
argument for a pull-back operator on the parameter space.

\end{abstract}

\maketitle

\tableofcontents

\section{Introduction}
A central goal in dynamical systems is to understand the global behavior of a system from qualitative or combinatorial data. Even for low-dimensional systems, this goal seems rather ambitious. For example, the maps in the Hénon family $H_{ab}(x,y)=(1-ax^2+y, bx)$ introduced in \cite{henon1} and \cite{henon2} are still not well understood, despite being the object of intense research in recent years, see \cite{benedicks}. But for systems acting in one-dimensional phase spaces, there is a long history of successful results, dating back to 1885, when Poincaré introduced the concept of {\em rotation number}. This single invariant is sufficient to describe combinatorial aspects of the dynamics of homeomorphisms of the circle, leading to their topological classification. The existence of periodic points is equivalent to the rotation number being rational, and a homeomorphism with irrational rotation number is semi-conjugate to the rigid rotation given by this number. In 1932, Denjoy proved that, for sufficiently smooth diffeomorphisms, this semi-conjugacy is indeed a conjugacy. In 1977, Milnor and Thurston \cite{milnor-thurston} (see also \cite{parry} and \cite{metropolis}) developed a combinatorial theory, today known as {\em  kneading
theory}, which gives a topological description and allows a
classification of the dynamics of multimodal maps, that is to
say, continuous endomorphisms of the interval with finitely many
turning points. From this theory, it follows that a $C^1$
unimodal map ({\it i.e.\/}, with just one turning point) is
semi-conjugate to  a quadratic map in the family $Q_{a}(x)= a x
(1-x)$,  where  $ a \in [0,4]$. This semi-conjugacy is a conjugacy if we assume that it is already a conjugacy when restricted to the
immediate basin of periodic attractors, and the unimodal map has no
wandering intervals, intervals of periodic points, or inessential
periodic attractors (that is, a periodic attractor whose immediate basin of attraction does not contain a turning point).
This result was generalized to the case of multimodal maps of the
interval by  W. de Melo and S. van Strien \cite{MS}, see also
\cite{salomao}.

In this paper, we construct, for each $m \ge 1$, a finite-dimensional family of rational maps whose restrictions to the unit circle are $2m$-multimodal circle maps.  For every post-critically finite $2m$-multimodal map $g$ satisfying certain conditions, see Definition \ref{defi_2m_modal}, we show that $g$ is topologically conjugate to a map in this family. We also prove a uniqueness result: two post-critically finite maps in this family that are topologically conjugate on the circle differ only by a rotation. 

On the path towards our goals, we follow some of the ideas in \cite{MS} while dealing with the additional difficulties posed by the topology and the cyclic order on the circle. It is important to point out here that the existence part of Theorem~\ref{Theo_Principal} was proved in \cite{MeloSalomaoVargas}   for  the family of maps $ p_{\mu}: \s \to \s $ induced by the trigonometric polynomials
\begin{equation}
P_{\mu}(t)=dt+ \mu_{2m}\sin(2 \pi mt)+ \sum_{j=0}^{m-1}( \mu_{2j}\sin(2\pi
jt) + \mu_{2j+1}\cos(2\pi jt)),
\end{equation} 
where $d \in \Z$ and  $ \mu =(\mu_1,\ldots,\mu_{2m}) \in \R^{2m}$ is in the 
set such that $p_\mu$ is $2m$-multimodal. It is proved in \cite{RempeStrien} that $p_\mu$ is unique in each class of topological conjugacy of $2m$-multimodal maps of the circle without periodic attractors. This type of result was also proved in \cite{kss} in the context of multimodal maps of the interval.

\section{Definitions and main results}

Let us start with some topological concepts related to multimodal maps of the 
circle. First of all, we observe that these maps have an even number of 
turning points, say $2m$, and in this case they are called $2m$-multimodal.
Below we define the {\em type}  $\tau$ of a  $2m$-multimodal map.

\begin{defi}\label{type} 
Let  $g: \s \to \s$  be a continuous $2m$-multimodal map and 
$c_1 < c_2 < \cdots < c_{2m} $ its turning points, 
ordered according to the counterclockwise orientation of the circle.
If  $c_1$ is a local  maximum,
the {\em type} of $g$ relative to $c_1$ is the vector
 $\tau=(\tau_1, \ldots, \tau_{2m-1} ),$  where 
 \[\tau_j := (-1)^j  \min \#\{[c_j, c_{j+1}] \cap g^{-1}(z) : z \in \s \},   \quad j=1, \ldots, 2m-1.\]  
 If $G: \R \to \R$ is a lift of $g$ with respect to 
 a covering map $\Pi: \R \to \s$  such that 
 $0 = C_1< C_2< \cdots < C_{2m} < 1 $ are the
 corresponding turning points, then
 $ \tau_j = (-1)^j \min \# \{[G(C_j), G(C_{j+1})]\cap \{s +\Z\}: s \in \R \}$
 and $\tau$  also defines the type of $G$ relative to $C_1.$ 
Moreover, the topological degree of $g$ is the integer $d$ such that 
$G(t+1) = G(t) +d$.
We say that $\kappa=(k_0, \ldots, k_{m})$  is above  $\tau$
and denote $\kappa \succ \tau $ if $k_j \geq -  \tau_{2j-1}+2, $ 
for $j=1, \ldots, m.$
\end{defi}

The class of $2m$-multimodal maps of the circle considered in this paper is given by the following definition. 

\begin{defi}\label{defi_2m_modal} Denote by $\mathcal G_m$ the set of all continuous $2m$-multimodal maps $g:\mathbb S^1 \to \mathbb S^1$ satisfying the following conditions:
\begin{itemize}
    \item[(i)] $g$ has no non-trivial intervals of periodic points of the same period;

    \item[(ii)] $g$ has no wandering intervals;

    \item[(iii)] $g$ has no inessential periodic orbits, i.e., attracting periodic orbits without a turning point in its immediate basin of attraction;

    \item[(iv)] every periodic turning point is attracting. 
    
\end{itemize}
\end{defi}

We remark that any $2n$-multimodal map of class $C^1$
is topologically semi-conjugate to a $2m$-multimodal  map 
in $\mathcal{G}_{m}$, for some $m \leq n$.
Moreover, the semi-conjugacy is non-injective only on the basins of inessential
periodic attractors,  wandering intervals, and non-trivial intervals of periodic points of the same period, 
see \cite{MS}.

We consider the family $f_{\mu \kappa}$  of maps on the circle $\s$
arising as restrictions to $\s$ of the Blaschke-type
products
\begin{equation} \label{F_a}
B_{\mu \kappa}(z)=e^{2\pi i
\eta_0}z^{k_0}\prod_{j=1}^{m}\left(\frac{z-a_j}{1-\overline{a}_jz}\right)^{k_j},
\end{equation}
where the parameter 
$\mu=(\eta_0, a_1, r_2, \eta_2, \ldots, r_m, \eta_m) \in \R^{2m}$ satisfies: 
(i) $\eta_0 \in \R$, (ii) $a_1>1$ and  
(iii) $a_j=r_je^{2\pi i \eta _j},$ where $r_j > 1$ and $\eta_j \in \R$, 
for $j=2,\ldots,m.$  
The  vector 
$\kappa= (k_0, \ldots, k_m)$ of positive integers $k_j$, which we  will fix and omit from the notation later on, controls the type and the topological  degree of the map $f_{\mu\kappa}$, 
while $\mu$ controls the position of its critical values on the circle. 

Although the formula for 
$B_{\mu \kappa}$ resembles that of a finite Blaschke product, we assume here that $|a_j|>1$, so $B_{\mu \kappa}$ is not a 
Blaschke product in the classical sense (where all zeros lie in the unit disk). Instead, it is a Blaschke-type product (or a Blaschke quotient), i.e., a rational map symmetric with respect to $\mathbb S^1$ and mapping $\mathbb S^1$
 to itself. Its restriction to 
$\mathbb S^1$ has at most $2m$ critical points, see Proposition~\ref{Riemann}. This implies that its restriction to $\s$, denoted $f_{\mu \kappa}:=B_{\mu \kappa}|_{\mathbb S^1}$, is at most $2m$-multimodal, in which case the critical points are quadratic. Therefore, we consider the non-empty set, as seen in Proposition~\ref{Riemann} and Lemma \ref{multimodal} below, given by
$$
\Delta :=\left\{ \mu =(\eta_0,\, a_1, \,r_2, \, \eta_2, \ldots, \, r_m,\eta_m) \in \R^{2m}:
f_{\mu \kappa} \textrm{ is } 2m-\textrm{multimodal} \right \}.
$$ 

We shall prove later that  $f_{\mu \kappa}$ is in $\mathcal G_m$ for every $\mu\in \Delta$. 
As in Definition \ref{type}, if $\tau=(\tau_1, \dots, \tau_{2m-1})$ is the type of a $2m$-multimodal
map $g$, we say that a vector $\kappa=(k_0, \ldots, k_m)$ is {\em above}
 $\tau$ and denote $\kappa \succ \tau$ if $k_j \geq -\tau_{2j-1}+2$,
 $j=1, \ldots, m$. Notice that this condition does not involve $k_0$ and $\tau_{2j}, j=1,\ldots m-1$. Indeed, $k_0$ is determined by the relation $k_0 = d+ \sum_{j=1}^m k_j$, where $d$ is the degree of $f_{\mu \kappa}$ as an endomorphism of $\mathbb S^1$. Hence, by choosing $k_i>0$ sufficiently large, we assume that 
 \begin{equation}\label{cond_kj}
 k_j \ \geq \ -\tau_{2j-1}+2, \quad \forall j=1,\ldots,m,\qquad \mbox{ and } \qquad k_0 \ =\  d +\sum_{j=1}^m k_j>1.
 \end{equation}

Our  main result guarantees that for $\kappa \succ \tau$,
the  family $f_{\mu \kappa}$, $\mu \in \Delta$,  exhibits
all the interesting dynamical behaviors of post-critically finite $2m$-multimodal maps on the circle with type
$\tau$. 

Given $g\in \mathcal{G}_m$, we denote by $\C_g\subset \mathbb{S}^1$ the (finite) set of turning points of $g$. We say that $g$ is post-critically finite if $\bigcup_{j=0}^\infty g^j(\C_g)$ is a finite set.

\begin{theo}\label{Theo_Principal}  Let   $g \in \mathcal{G}_m$ be a 
post-critically finite $2m$-multimodal map of type $\tau$ and degree $d$.
Fix $\kappa=(k_0,\ldots,k_m)$ satisfying \eqref{cond_kj}. Then there exists a
parameter $\mu \in \Delta$ such
that  $g$ is topologically conjugate to $f_{\mu \kappa}$. Moreover, if $f_1 = B_{\mu_1 \kappa}|_{\mathbb S^1}$ and $f_2=B_{\mu_2 \kappa}|_{\mathbb S^1}$ are both conjugate to $g$ for some $\mu_1, \mu_2 \in \Delta$, then $B_{\mu_1 \kappa}$ and $B_{\mu_2 \kappa}$ coincide up to conjugation by a rotation of the complex plane.   
\end{theo}

To prove Theorem~\ref{Theo_Principal}, first we show that all
types $\tau$ such that $\kappa  \succ \tau$ can be realized
in the family $f_{\mu\kappa} $. Then we follow the strategy in  \cite{MS} to
show that this family
realizes all combinatorics of post-critically finite $2m$-multimodal maps.
This is the main step to get the existence
part of the theorem and depends on showing that a certain operator
acting in the space $\Delta$ of the parameter $\mu$ has a fixed point.
Thurston introduced this operator to  show that  post-critically
finite branched coverings of the 2-sphere $\mathbb S^2$ with hyperbolic orbifold and satisfying a combinatorial condition can be realized by rational maps on the Riemann sphere, see
\cite{mcmullen}. Following the same approach as in \cite{MS, salomao}, it follows that the case where the post-critical set is infinite and
the $\omega$-limit set of $\C_g$ is finite can also be realized by
the same strategy of the proof of Theorem~\ref{Theo_Principal}
by choosing convenient finite pieces of
critical orbits. The case where the $\omega$-limit set of $\C_g$ is infinite can be realized by taking the limit of maps from these
previous cases. In this paper, however, we only address the case of finite combinatorics. 

\subsection{Uniqueness} The uniqueness part of Theorem \ref{Theo_Principal} depends on the following
theorem.

\begin{theo} For a fixed $\kappa$, if   $f_{\mu_1 \kappa}$ and 
$f_{\mu_2 \kappa}$
are  topologically conjugate  post-critically finite  $2m$-multimodal
maps of the circle, then the Blaschke-type products  
$B_{\mu_1 \kappa}$ and $B_{\mu_2
\kappa}$ are Thurston equivalent.
\end{theo}

This theorem, together with a theorem of Thurston, which is stated below
(see \cite{mcmullen} for a very elegant treatment),  implies that, for
a fixed $\kappa$, two topologically conjugate post-critically finite
$2m$-multimodal  maps in the family $f_{\mu\kappa}$ are the same up to
conjugation by a rotation of the circle.

\begin{theo}[Thurston]\label{thurston}
A post-critically finite branched covering of the $2$-sphere with a hyperbolic orbifold is Thurston equivalent to at most one rational map
up to conformal conjugacy.
\end{theo}

Recall that two post-critically finite branched coverings $\Psi_1$
and $\Psi_2$ of the $2$-sphere are
said to be {\it Thurston equivalent} if there exist homeomorphisms $H_0, \, H_1: \mathbb S^2\to \mathbb S^2$ such that
$$
\begin{CD}
(\mathbb S^2,P_{\Psi_1})@>{H_1}>>(\mathbb S^2,P_{\Psi_2})\\
@V{\Psi_1}VV             @VV{\Psi_2}V\\
(\mathbb S^2,P_{\Psi_1})@>{H_0}>>(\mathbb S^2,P_{\Psi_2})
\end{CD}
$$
commutes, and $H_1$ is isotopic to $H_0$ relative to $P_{\Psi_1}$.
Here, $P_{\Psi_i} = \bigcup_{n=1}^{\infty}\Psi_i^n(\C_{\Psi_i})$
is the post-critical set of $\Psi_i$, $\ \ i=1, 2$.

Recall also that an {\it orbifold\/} is a Hausdorff topological
space locally modeled on the quotient of the Euclidean space by a finite group of diffeomorphisms. In the statement of
Theorem~\ref{thurston} above, only a two-dimensional version of this notion is needed. It turns out, see \cite{mcmullen1}, that a two-dimensional orbifold with underlying space $X$ is determined by the following data:
\begin{itemize}
\item[(a)] A function $N:X\to \mathbb{N}$ such that the set $\{x\in
X: N(x)>1\}$ is discrete;
\item[(b)] An open cover $\{U_\alpha\}$ of $X$ with the property
that for each $x\in X$ and each $U_\alpha\ni x$ there exist a
homeomorphism $\psi_\alpha:\mathbb{D}\to U_\alpha$ with
$\psi_\alpha(0)=x$ and a branched covering map
$\phi_\alpha:\mathbb{D}\to U_\alpha$ such that
$\phi_\alpha(z)=\psi_\alpha(z^n)$, where $n=N(x)$.
\end{itemize}
The Euler characteristic of an orbifold $\mathcal{O}=(X,N)$  is the
rational number
\begin{equation} \label{orbifold}
\chi(\mathcal{O})\;=\;\chi(X) - \sum_{x\in
X}\left(1-\frac{1}{N(x)}\right) \ .
\end{equation}
An orbifold is said to be {\it hyperbolic\/} if its Euler
characteristic is negative. Associated with each  critically finite
branched covering $\Psi:\mathbb{S}^2\to \mathbb{S}^2$ we have an
orbifold with underlying space $X=\mathbb{S}^2$ (or $\mathbb{S}^2$
minus finitely many points) provided we take $N(z)$ to be the
least common multiple of the local degrees $\deg(\Psi^n, w)$ at
those points $w\in \mathbb{S}^2$ such that $\Psi^n(w)=z$. Notice that
$N(z) > 1$ if and only if $z \in P_{\Psi}$. It may be the case that
$N(z) = \infty$ for some points; this happens precisely when there are
periodic critical points. These points must be excluded. In any
case, if the post-critical set $P_{\Psi}$ has more than four points, then the orbifold of $\Psi$ is hyperbolic, as is easily seen from
\eqref{orbifold}.

In our setting, the orbifold associated to $B_{\mu \kappa}$ is always
hyperbolic. To see that, observe that $0$ and $\infty$ are critical
and fixed points of $B_{\mu \kappa}$. This implies that
$N(0)=N(\infty)=\infty$ and, therefore, $X\subset
\mathbb{S}^2\backslash \{0,\infty\}$. We conclude that $\chi(X)\leq
0$ and, therefore, $\chi(\mathcal{O})\leq0$. The post-critical set
contained in $\s$ may further decrease the Euler characteristic
(this is the case when there is a periodic critical point
$c\in\s$ implying $N(c)=\infty$ and $c\notin X$), or has
a point satisfying $1<N(x)<\infty$ (this happens when there is a
critical point in $\s$ which is eventually periodic but
not periodic). In both situations we see from (\ref{orbifold}) that
the contribution of the post-critical set inside $\s$
implies $\chi(\mathcal{O}) < 0$.

\section{Blaschke-type multimodal maps}

The map $B= B_{\mu\kappa}
:\overline{\mathbb{C}} \to \overline{\mathbb{C}}$ given by
{\eqref{F_a}}  has  degree $d_B=\sum_{j=0}^m k_j$ 
on the Riemann sphere
 $\overline{\mathbb{C}}$  and is symmetric with respect to  the unit circle $\s,$ 
 that is $B(1/ \overline{z})= 1/\overline{B(z)}$.
A simple computation shows that  the derivative of $B$ is:
\begin{equation} \label{fl_a}
B^{\, \prime}(z)\;=\; B(z)\left (\frac{k_0}{z}+
\sum_{j=1}^{m}\,k_j\,\frac{1-|a_j|^2}{(1-\overline{a}_jz)(z-a_j)} \right).
\end{equation}
Thus, if $\C_B$ is the set of all critical points of $B$,
the Riemann-Hurwitz formula, see \cite{reyssat}, tells us that
\begin{equation} \label{hurwitz}
\sum_{c \ \in\  C_B} \nu(B;c)\;=\;2d_B-2 \; = \; 2\sum_{j=0}^m k_j -2\ ,
\end{equation}
where $\nu(B;c)=\mbox{deg}(B;c)-1$ is the {\it discrepancy\/} or
{\it defect\/} of $B$ at $c$. Note that the defect of $B$ at a point
which is not critical is zero. Hence, we can think of the sum on the
left-hand side of \eqref{hurwitz} as extended over all points of the
Riemann sphere. Now, from {\eqref{F_a}} and (\ref{fl_a})  we see that

$$
\C_B \ \supset
\left\{a_1,\,\frac{1}{ a_1}, \,a_2, \, \frac{1}{\overline{a}_2}, \ldots, a_m,
\, \frac{1}{\overline{a}_m}\right\} \, \cup\, \{0,\infty\} \ .
$$
Moreover

\begin{equation} \label{crit1}
\nu(B;0)\;=\;\nu(B;\infty)\;=\;k_0-1 \ \ \textrm{ and } \ \
\nu(B;a_j)\;=\;\nu(B; 1/\overline{a}_j)\;=\;k_j-1
\end{equation}
for all $j=1,\ldots,m$. These facts yield the following proposition.

\begin{pro}\label{Riemann}
The restriction of $B$ to the circle  $\s$  has
at most $2m$ critical points, in which case these critical
points are quadratic, and the set $\Delta$ is open.
\end{pro}
\begin{proof}
Combining  the Riemann-Hurwitz formula {\eqref{hurwitz}} with
{\eqref{crit1}}, we get
$$
2(k_0-1)+2\sum_{j=1}^m (k_j-1) +\sum_{c \ \in \ \C_B \ \cap \  \s}
\nu(B;c) \; \leq \;2\sum_{j=0}^m k_j\,-2 \ ,
$$
and therefore $\sum_{c \, \in\,  \C_B \, \cap \,  \s} \nu(B;c)\leq 2m$. This
shows that $B$ has at most $2m$ critical points on the  circle.
If $B$ has precisely $2m$ such critical points, then necessarily
$\nu(B;c)=1$ for all $c \ \in \  \C_B \ \cap \  \s$. Therefore, all
critical points of $B$ on the circle are quadratic turning points.
\end{proof}

%\begin{remark}Proposition \ref{Riemann} can be generalized in the following way: for every $k_0\in \Z_{\geq 0}$, $k_1,\ldots,k_m\in \Z_{\geq 0}$ and parameters $\eta_0\in \R, |a_j|>1, j=1,\ldots,m$, the number of critical points on the unit circle $\s$ of the corresponding map $B$ is at most $2M$, where $M:=\#\{j\in \{1,\ldots,m\}: k_j > 0\}$. 
%\end{remark}

The next lemma states that there are parameters 
$\mu=(\eta_0, \,a_1,\, r_2, \,\eta_2, \ldots, r_m,\, \eta_m) \in \R^{2m} $ 
such that 
$f_{\mu\kappa}$ (the restriction of $B_{\mu\kappa}$ to the circle)  is 
 $2m$-multimodal. 
For this, remember that  $a_j = r_je^{2\pi i \eta_j}, j=1, \ldots,m$, with $r_j>1,$ and  $a_1 =r_1>1$ (we can choose 
$\eta_1 =0$). We 
define $r_{\max}=\max\{r_1, \ldots, r_m\}$ and define 
$r_{\min}=\min \{r_1, \ldots, r_m \}$.

\begin{lem}\label{multimodal} If
$0= \eta_1 < \eta_2 < \cdots < \eta_m < 1, $ then there exists $\epsilon >1$
such that if $1 < r_{\min} \leq r_{\max} < \epsilon$, then $f_{\mu\kappa}$ is a
$2m$-multimodal map. %Moreover, if $r_{\min} <1$, then  $f_{\mu \kappa}$ is at most $2(m-1)$-multimodal.
\end{lem}

\begin{proof}
From (\ref{F_a}) we see that $f = f_{\mu\kappa}$ satisfies
\begin{equation} \label{f_a_2}
f(z)=e^{2\pi i
\eta_0}z^{k_0-\sum_{j=1}^{m}k_j}\prod_{j=1}^{m}\left(\frac{z-a_j}{|z-a_j|}\right)^{2k_j}, \qquad \forall z\in \mathbb S^1.
\end{equation}
Let $F : \R \to \R$ be a lift  of $ f$ with
respect to the covering $\Pi(t)=  e^{2\pi i t}$, that is,  $e^{2 \pi
i F(t)}=f(e^{2 \pi i t})$. From (\ref{f_a_2}) it follows that up to an integer
\begin{equation}\label{theta_til}
F(t)=\eta_0 + k_0t + \sum_{j=1}^{m}k_j(2\varphi_{j}(t)-t)
\end{equation}
where $\varphi_{j }(t)$ is an analytic function satisfying
$\varphi_j(\eta_j)=\eta_j+\frac{1}{2}$ and
\begin{equation} \label{varfi}
e^{2 \pi i \varphi_j (t)}=\frac{e^{2 \pi i t}-a_j}{|e^{2\pi i t}-a_j|}\cdot
\end{equation}

For $ j=1,\ldots, m$  define $z_j = \frac{a_j}{|a_j|}=e^{2 \pi i
\eta_j}$ and $z_{m+1} = z_1$. Then let $(z_j , z_{j+1})$ be the arc
on $\s$ which does not intersect $\{z_1, \ldots, z_m \}$.
Choose $s_j \in [0,1)$ such that $w_j := e^{2 \pi i s_j}
\in (z_j, z_{j+1})$. A direct computation shows that
\begin{equation}\label{efelinha}
F'(t)=k_0-\sum_{j=1}^{m}k_j\frac{r_j^2-1}{| e^{2\pi i t}- a_j|^2} \  \cdot
\end{equation}
Therefore we have, for all $n=1,\ldots, m$,
\begin{equation} \label{eq.Fleta}
F'(\eta_n)=k_0-\sum_{j=1,  j\neq n}^{m}k_j\frac{r_j^2-1}{| z_n-
a_j|^2}-k_n\frac{r_n+1}{ r_n-1} \ ,
\end{equation}
which tends to $-\infty$ as $r_n\to 1^{+}$. Likewise, we have, for all
$n=1,\ldots, m$,
\begin{equation} \label{eq.Fls}
F'(s_n)=k_0-\sum_{j=1}^{m}k_j\frac{r_j^2-1}{| w_n- a_j|^2} \ ,
\end{equation}
which tends to $k_0 > 0$ as $r_n \to 1^{+}$. It follows that there exists
$\epsilon>1$ such that if $1 < r_{\min}\leq r_{\max} < \epsilon,$ then $F^{\prime}(\eta_n)<0$ and
$F^{\prime}(s_n)>0$. This implies the lemma. 
\end{proof}

Proposition~\ref{Riemann} and Lemma~\ref{multimodal}  above imply that, for  $\kappa$ fixed, the set of parameters $\Delta =\{\mu\in \R^{2m}: f_{\mu\kappa} \; \textrm{ is } \; 2m-\textrm{multimodal} \}$ is a non-empty open set.  

The topological lemma below, illustrated in Figure~\ref{fig:julia},  gives a description of the pre-image of the unit circle under a Blaschke-type
product $B_{\mu\kappa}$ considered in equation \eqref{F_a}.
It is needed to prove that the turning points of a
lift of $f_{\mu \kappa},$ are 
$2m$ bounded analytic functions
globally defined in $\Delta$, see Proposition~\ref{prop:lifts}.
It is also needed for the application of Thurston's criterion,  given in Theorem~\ref{thurston},
to prove the uniqueness part of Theorem~\ref{Theo_Principal}.
 
\begin{lem} \label{toplemma2}
If the restriction of $B = B_{\mu\kappa}$ to the unit circle $\s$ is 
$2m$-multimodal, then
\begin{equation}\label{preimage}
B^{-1}(\s)\;=\;\s \cup \Gamma_1 \cup \Gamma_2 \cup  \ldots \cup \Gamma_m, 
\end{equation}
where each $\Gamma_j$ is an analytic Jordan curve, and the following
properties hold:
\begin{itemize}
\item[(a)] Each $\Gamma_j$ is symmetric under inversion about $\s$;

\item[(b)] We have $\Gamma_j\cap \s =\{c_j',c_j''\}$, where
$c_j', \ c_j'' \in \s $ are distinct critical points of $B$;

\item[(c)] If $D_j$ is the  topological open disk bounded by $\Gamma_j$,
then the closures of $D_1, \ldots, D_m$ are pairwise disjoint and  
$ D_j $ contains $a_j$ and   $1 / \overline{a}_j.$
\end{itemize}
\end{lem}

\begin{figure}[ht]
\includegraphics[width=8cm, height=8cm]{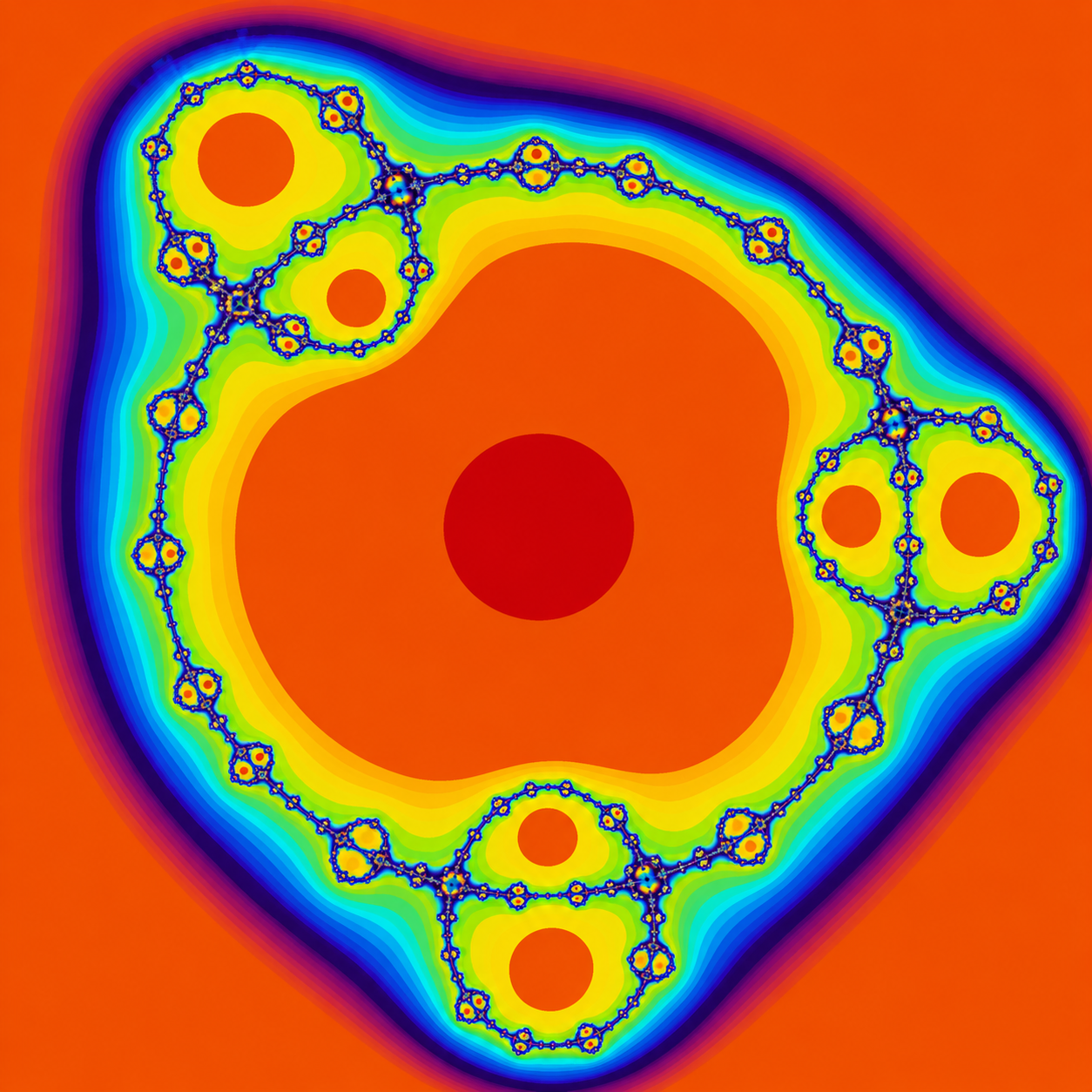}
\caption{The Julia set of   $B_{\mu\kappa}$ for $m=3$, $\kappa=(8,3,2,2)$, and $\mu \in \Delta$ with $a_1 = 1.2$, $a_2=1.2e^{2\pi i /3}$ and $a_3 =1.1e^{3\pi i /2}.$ The pre-image of $\s$ under $B_{\mu \kappa}$ contains $\s$ and three simple closed curves $\Gamma_j, j=1,2,3,$ each enclosing $a_j$ and $1/\overline{a}_j$, and intersecting $\s$ at two critical points of $f_{\mu \kappa}=B_{\mu \kappa}|_\s$.  }
 \label{fig:julia}
\end{figure}

\begin{proof} Let $c\in \s$ be a critical point of
$B$, and let $v\in \mathbb S^1$ be its critical value. Let $J \subset \s$ 
be a small open arc containing $v$. Since $c$ is a
quadratic critical point, if $\hat D_0$ is a sufficiently small disk
centered at $c$, we can write
$$
B^{-1}(J)\cap \hat D_0\;=\;\alpha \cup \beta \ ,
$$
where $\alpha\subset \s$ and $\beta$ are analytic
arcs, $\alpha \cap \beta =\{c\}$ and $\alpha$ and $\beta$ meet at
right angles at $c$. Let $\beta'$ be the part of $\beta$ lying
inside the open unit disk $\mathbb{D}.$ Using the monodromy theorem, we can
analytically continue $\beta'$ away from $c$ inside $ \mathbb{D}$,
without ever reaching a critical point of $B$ inside the disk,
because no such critical point is mapped to $\s$.
The maximal arc $\gamma \subset \mathbb{D}$ obtained in this way
must therefore hit the boundary $\s$ again, at
another point $c'$ distinct from $c$. The point $c'$ is necessarily
a critical point of $B$. If we join $\gamma$ to its symmetric image
(under inversion) outside $\mathbb{D}$, we get a Jordan curve
$\Gamma$ intersecting $\s$ at $\{c, c' \}$.  Now let
$D$ be the topological open disk in the plane bounded by $\Gamma$. Note
that $D$ is symmetric under inversion about $\s$,
and therefore it does {\it not\/} contain the origin. Since $B$ is
an open map, both $D \cap \mathbb{D}$ and $D\cap (\mathbb{C}\setminus
\overline{\mathbb{D}})$ map {\it onto\/} a component of
$\overline{\mathbb{C}}\setminus \s$, and they map onto
distinct components. Hence, one contains a pre-image of $0$, the
other a pre-image of $\infty$. By symmetry, we deduce that there
exists $1 \leq j \leq m$ such that $a_j \in D \cap (\mathbb{C}\setminus
\overline{\mathbb{D}})$ and $1 / \overline {a}_j \in D \cap \mathbb{D}$.

Next, suppose that $\tilde{c} \in \s$ is another
critical point of $B$, distinct from both $c$ and $c'$. Performing
for $\tilde{c}$ the same steps as above, we get another topological
disk $\tilde{D}$ with boundary $\tilde{\Gamma}$. The curves $\Gamma$
and $\tilde{\Gamma}$ are disjoint: indeed, if they crossed
somewhere, it would have to be outside $\s$, and the
crossing point would be a critical point of $B$ with critical value
in $\s$, but that is impossible. Therefore, either
$D \cap \tilde{D}=\emptyset$ or else one of these disks contains the
other, say $D\subset \tilde{D}$. To rule out the latter, let
$U=D\cap \mathbb{D}$ and $V=(\tilde{D}\cap \mathbb{D})\setminus U$.
Then both $U$ and $V$ are (adjacent) topological disks whose
boundaries are mapped onto $\s$. Since $U$ is mapped
onto $\mathbb{C}\setminus \overline{\mathbb{D}}$, we see that $V$ is
mapped onto $\mathbb{D}$. In particular, $V$ contains a pre-image of
$0$, which is impossible because $V \subset \mathbb{D}$ (and $0 \notin
V$). We deduce that the closures of the disks $D$ and $ \tilde{D}$ are disjoint.

From these arguments, it follows that $B^{-1}(\s)$
contains $m$ Jordan curves that are pairwise disjoint and meet $\s$ transversely at two critical points each. Each
such curve surrounds a pair of points of the form $\{a_j,
1 / \overline{a}_j\}$. Since there are exactly $m$ such pairs, we
can label the curves $\Gamma_1, \ldots, \Gamma_m$ and the
disks they bound $D_1, \ldots, D_m$, so that $D_j \supset  \{a_j,
1 / \overline{a}_j\}$. We have thus
\begin{equation}\label{uniq2}
B^{-1}(\s)\;\supseteq\; \s\,\cup
\,\bigcup_{j=1}^m \Gamma_j \ .
\end{equation}

We need to prove that this inclusion is, in fact, an equality. This is done by counting preimages of a regular value on the circle. Let $w\in \mathbb S^1$ be a regular value of $B$. Since the restriction $B|_{\mathbb S^1}$ has degree $d$, the point $w$ has exactly $d$ preimages on $\mathbb S^1$, counted with multiplicity.
On the other hand, for each $j=1,\dots,m$, the restriction
\[
B|_{D_j\cap \mathbb D}: D_j\cap \mathbb D \to \overline{\mathbb C}\setminus \overline{\mathbb D}
\]
is a branched covering of degree $k_j$, branched only at $1/\overline{a_j}$, while
\[
B|_{D_j\cap (\overline{\mathbb C}\setminus \overline{\mathbb D})}: D_j\cap (\overline{\mathbb C}\setminus \overline{\mathbb D}) \to \mathbb D
\]
is a branched covering of degree $k_j$, branched only at $a_j$. Therefore each of the two pieces $D_j\cap \mathbb D$ and $D_j\cap (\overline{\mathbb C}\setminus \overline{\mathbb D})$ contributes exactly $k_j$ preimages of $w$, counted with multiplicity. Summing over $j$, we find that the set
\[
\mathbb S^1\cup \Gamma_1\cup\cdots\cup \Gamma_m
\]
already contains
\[
d+2\sum_{j=1}^m k_j
\]
preimages of $w$, counted with multiplicity.

Now, by definition,
\[
d=k_0-\sum_{j=1}^m k_j
\qquad\text{and}\qquad
d_B=\sum_{j=0}^m k_j,
\]
hence
\[
d+2\sum_{j=1}^m k_j
=
k_0-\sum_{j=1}^m k_j+2\sum_{j=1}^m k_j
=
k_0+\sum_{j=1}^m k_j
=
d_B.
\]
Since $d_B$ is the degree of $B$ on $\overline{\mathbb C}$, this accounts for all preimages of the regular value $w$. Therefore no additional components of $B^{-1}(\mathbb S^1)$ can exist, and we conclude that
\[
B^{-1}(\mathbb S^1)=\mathbb S^1\cup \Gamma_1\cup\cdots\cup \Gamma_m.
\]
This proves the lemma.
\end{proof}

Proposition~\ref{Riemann} implies that  the critical points 
of $f_{\mu\kappa}$  depend analytically on $\mu \in \Delta$ 
and, as a corollary of Lemma~\ref{toplemma2}, we conclude that they are analytic functions which are globally well defined and labeled according to the counterclockwise order on $\s$ in each connected component of $\Delta$. Later on, this fact will be fundamental to guarantee a diffeomorphic correspondence between parameters in a connected component of $\Delta$ and vectors of turning points in an appropriate simplex. As a consequence of
this, we will show that $\Delta$ is simply connected.

\begin{corollary}\label{label} 
For $\mu$ and $\kappa$  fixed,  if $\textrm{cc}(\Delta)$ is a
connected component of $\Delta$, then the $2m$ critical points
of $f_{\mu \kappa}$ are well-defined analytic functions 
$c_1, \ldots, c_{2m}: \textrm{cc}(\Delta) \to \s$, and can be labeled according to the cyclic order on $\s$, satisfying $c_{2j-1}< c_{2j} \in \Gamma_j, \forall j=1,\ldots,m$. Moreover, $c_{2j-1}$ and $c_{2j}$ are, respectively, points of maximum and minimum according to the cyclic order on $\s$. 
\end{corollary}

\begin{proof} First remember that 
 $\mu=(\eta_0, \,a_1,\, r_2, \,\eta_2, \ldots, r_m,\, \eta_m) \in \R^{2m}$ satisfies: 
$\eta_0\in \R, a_1 >1 $ and  $a_j = r_je^{2\pi i \eta _j},$ where 
$r_j > 1$ and $\eta_j \in \R$,  for $j=2,\ldots,m \, .$  
According to Lemma~\ref{toplemma2}, if $\mu \in \Delta$, the inverse image
$B_{\mu \kappa}^{-1}(\s)$ contains $m$ Jordan curves. Each of them meets
$\s$ transversely at two critical points of $f_{\mu \kappa}$ and
bounds a topological disk 
which contains the points $\{a_j,
1 / \overline {a}_j\}$.  Since there are exactly $m$ such pairs of points, we
can label the curves  $\Gamma_1(\mu), \ldots, \Gamma_m(\mu)$ and the
disks they bound $D_1(\mu), \ldots, D_m(\mu)$, so that 
$D_j(\mu) \supset  \{a_j, 1 / \overline{a}_j\}$. This implies that the $2m$ critical points given
by the intersections
$\s \cap \Gamma_1(\mu), \ldots, \s \cap \Gamma_m(\mu)$ are 
globally defined analytic functions in $cc(\Delta)$. Hence, they can be globally labeled according to the cyclic order on $\s$ as
$c_1(\mu) < c_2(\mu)< \cdots < c_{2m-1}(\mu)< c_{2m}(\mu)$ so that  $c_{2j-1}$ and $c_{2j}$ are contained in $\Gamma_j$ for every $j=1,\ldots, m$. 

The last statement of the corollary follows from the facts that 
the arc  $[c_{2j-1}(\mu), c_{2j}(\mu)] \subset \Gamma_j$ is part of the boundary of the disk $D_j \setminus \mathbb{D}$ and  $B_{\mu \kappa}$ maps this disk onto the unit disk as a branched covering of degree $k_j$.
\end{proof}

According to Corollary \ref{label}, for $\mu$ in each connected component  of $\Delta$, the label 
$c_1(\mu) < c_2(\mu) < \cdots < c_{2m-1}(\mu) < c_{2m}(\mu)$ 
of the critical points of $f_{\mu \kappa}$ is chosen 
in such a way that 
$c_{2j-1}(\mu) < c_{2j}(\mu)$ are the points in  $\s \cap \Gamma_j(\mu)$.
Moreover, $c_{2j-1}(\mu)$ and $c_{2j}(\mu)$ are, respectively,  points of maximum and minimum according to the cyclic order on the circle.

\begin{pro}\label{prop:lifts}
Let $\mathrm{cc}(\Delta)$ be a connected component of $\Delta$.  
Let 
$c_1(\mu)<\dots <c_{2m}(\mu)\in\mathbb{S}^1, \mu \in cc(\Delta),$ denote the critical points of 
$f_{\mu\kappa}=B_{\mu\kappa}|_{\mathbb{S}^1}$, labeled according to their cyclic order as in Corollary \ref{label}.
Then each $c_i(\mu)$ admits a globally defined continuous (indeed, analytic) lift
$$
C_i:\mathrm{cc}(\Delta)\longrightarrow\mathbb{R}, \qquad
e^{2\pi i\, C_i(\mu)}=c_i(\mu),
$$ satisfying $C_1(\mu) < C_2(\mu)< \ldots <C_{2m}(\mu) < C_1(\mu)+1,$ for every $\mu \in cc(\Delta).$ 
In particular, no lift undergoes an integer jump when the parameter $\mu$ traverses a loop in $\mathrm{cc}(\Delta)$.
\end{pro}

\begin{proof}
By Lemma~\ref{toplemma2}, for each $\mu\in\mathrm{cc}(\Delta)$ the preimage  $B_{\mu\kappa}^{-1}(\mathbb{S}^1)$ is the union of $\mathbb{S}^1$ with $m$ analytic Jordan curves 
$\Gamma_j(\mu)$, where $\Gamma_1(\mu)$ bounds a disk $D_1(\mu)$ containing $a_1(\mu)$.  
Since the family $B_{\mu\kappa}$ depends analytically on $\mu$ and the critical points in $\Gamma_j(\mu)$ are nondegenerate,  
the curves $\Gamma_j(\mu)$, the disks $D_j(\mu)$, and the critical points 
$c_i(\mu)=\Gamma_j(\mu)\cap\mathbb{S}^1$ all depend analytically on $\mu$.

Because $a_1(\mu)$ is real and positive for all $\mu\in\mathrm{cc}(\Delta)$, we may select on each disk $D_1(\mu)$ the unique holomorphic branch of the logarithm satisfying
\[
\log_{D_1(\mu)}(a_1(\mu))=\ln(a_1(\mu))\in\mathbb{R}.
\]
This branch varies continuously with $\mu$ and induces continuous argument functions on the two points of 
$\Gamma_1(\mu)\cap\mathbb{S}^1$.  
Along any loop $\gamma:[0,1]\to\mathrm{cc}(\Delta)$ with $\gamma(0)=\gamma(1)$, the chosen branch 
$\log_{D_1(\gamma(t))}$ begins and ends at the same value, since $a_1(\gamma(t))>1$ for all $t$.  
Hence, the arguments of the two corresponding critical points return to their initial values. In particular,  
they cannot change by a multiple of $2\pi$.

This proves that each of the two critical points arising from $\Gamma_1(\mu)$ admits a global continuous lift 
$C_i(\mu)\in\mathbb{R}$, uniquely determined by the normalization above.  
The same argument applies to the remaining curves $\Gamma_j(\mu)$ after fixing the 
labeling of critical points on $\mathbb{S}^1$. Hence, all $2m$ critical points admit global analytic lifts.
\end{proof}

Fix $\kappa = (k_0, k_1,\ldots, k_m)$ and the connected component $cc(\Delta) \subset \Delta$. Consider the lifts $C_1(\mu) < \ldots < C_{2m}(\mu)< C_1(\mu)+1, \mu \in cc(\Delta),$ of the critical points of $f_{\mu \kappa}$, as in Proposition \ref{prop:lifts}. Recall that $C_{2j-1}(\mu)< C_{2j}(\mu)$ are points of maximum and minimum of $F_{\mu \kappa}$, respectively, for every $j=1,\ldots, m$. The next lemma guarantees that
$F_{\mu \kappa}(C_{2j-1}(\mu) )- F_{\mu \kappa}(C_{2j}(\mu) ) $
can be as close to $k_j$ as we wish, but not larger than or equal to $k_j$. 

\begin{lem}\label{lem:gap-kj}
Let $\mu\in cc(\Delta) \subset \Delta$ and let $F_{\mu\kappa}:\mathbb{R}\to\mathbb{R}$ be a lift 
of $f_{\mu\kappa}$ with respect to $\Pi(t)=e^{2\pi i t}$. Let
$
 C_1(\mu)<C_2(\mu)<\cdots<C_{2m}(\mu)<C_1(\mu)+1
$
be the critical points of $F_{\mu\kappa}$ as in Corollary~\ref{label}. 
Then, for each $j=1,\ldots,m$, we have
$
 0 \;<\; 
 F_{\mu\kappa}(C_{2j-1}(\mu)) - F_{\mu\kappa}(C_{2j}(\mu))
 \;<\; k_j.
$
\end{lem}

\begin{proof}
As in Lemma~\ref{toplemma2}, let $D_j(\mu)$ be the topological disk bounded by 
$\Gamma_j(\mu)$ and containing $a_j(\mu)$ and $1/\overline{a}_j(\mu)$. 
The restriction of $B_{\mu\kappa}$ to
$
 \hat D_j(\mu):=
 D_j(\mu)\cap\{\,|z|> 1\,\}
$
is a holomorphic map
$
B_{\mu\kappa}:\hat D_j(\mu)\to\mathbb{D}
$
of degree $k_j$, branched only at $a_j(\mu)$. In particular,
if we orient the boundary $\hat \Gamma_j(\mu):=\partial \hat D_j(\mu)$ positively, then the argument of 
$B_{\mu\kappa}$ strictly increases by $2\pi k_j$ along one full turn around $\hat \Gamma_j(\mu)$.

We decompose $\hat \Gamma_j(\mu)
 =
 \gamma_{j,1}(\mu) \cup \gamma_{j,2}(\mu),
$
where $\gamma_{j,1}(\mu)\subset \s$ is the arc of the unit circle from $c_{2j}(\mu)$ to $c_{2j-1}(\mu)$, with the induced orientation from 
  $\partial \hat D_j(\mu)$,
and $\gamma_{j,2}(\mu)$ is the part of $\hat \Gamma_j(\mu)$ lying outside the unit disk,
  joining $c_{2j-1}(\mu)$ to $c_{2j}(\mu)$. Notice that the only critical points of $B_{\mu \kappa}$ in $\gamma_{j,1}(\mu)$ and $\gamma_{j,2}(\mu)$ are the endpoints $c_{2j-1}(\mu)$ and $c_{2j}(\mu)$. Along $\gamma_{j,1}(\mu)$, we have $B_{\mu\kappa}(e^{2\pi i t}) = e^{2\pi i F_{\mu\kappa}(t)}.
$
Notice that the orientation of $\gamma_{j,1}(\mu)$ induced by the positive orientation of $\s$ is opposite to the orientation induced by $\hat \Gamma_j(\mu)$. Thus, the variation of the argument of $f_{\mu\kappa}$ along 
$\gamma_{j,1}(\mu)$ is equal to
\[
 \Delta{\gamma_{j,1}(\mu)} 
 :=
 2\pi(
 F_{\mu\kappa}(C_{2j-1}(\mu))
 -
 F_{\mu\kappa}(C_{2j}(\mu)))>0,
\]

The remaining part $\gamma_{j,2}(\mu)$ is a nontrivial analytic arc
of $\Gamma_j(\mu)$ intersecting $\s$ only at $c_{2j-1}(\mu)$ and $c_{2j}(\mu)$. Its image under $B_{\mu \kappa}$ is contained in $\s$ and the total variation of the argument of $B_{\mu \kappa}$ along $\gamma_{j,2}(\mu)$ is  $\Delta \gamma_{j,2}(\mu)>0$. Since  $\Delta{\gamma_{j,1}(\mu)}+  \Delta \gamma_{j,2}(\mu)=2\pi k_j$, we conclude from the estimates above that $0<F_{\mu\kappa}(C_{2j-1}(\mu))
 -
 F_{\mu\kappa}(C_{2j}(\mu))<k_j.$
\end{proof}

\begin{lem}\label{lem:rj-bounded}
Fix $m\ge1$ and $\kappa=(k_0,\ldots,k_m)$, and let $\Delta$ be the set of
parameters $\mu$ for which $f_{\mu\kappa}$
is $2m$-multimodal. Let $\mathrm{cc}(\Delta)$ be a connected component of
$\Delta$. Then there exists $R>0$ such that
$r_j(\mu)=|a_j(\mu)| \leq R$ for every $\mu \in {\rm cc}(\Delta).$
\end{lem}

\begin{proof}
Assume by contradiction that  there exists a non-empty subset $J \subset \{1,\ldots,m\}$ and a sequence $\mu^n, n\in \mathbb{N},$ of parameters in $cc(\Delta)$ such that
$r_j(\mu^n) \to +\infty$ as $n\to\infty$ for every $j\in J$. Denote $a_j^n := a_j(\mu^n)$ and
$B_n := B_{\mu^n\kappa}$, $f_n := f_{\mu^n\kappa} = B_n|_{\mathbb{S}^1}$. Recall that $B_n$ contains the factor $H_n(z):=\Pi_{j\in J}((z-a_j^n)/(1-\overline {a}_j^nz))^{k_j}$, which, up to a subsequence, converges uniformly on compact subsets of $\mathbb C\setminus \{0\}$ to the holomorphic function $H_\infty(z):= \Pi_{j\in J} e^{4\pi i k_j\eta_j^\infty}z^{-k_j},$ where $a_j^n = |a_j^n|e^{2\pi i \eta_j^n}$ and $\eta_j^n \to \eta_j^\infty$ as $n\to \infty$. We may assume for simplicity that $r_i^n\to r_i^\infty \in [1,+\infty)$ and $\eta_i^n \to \eta_i^\infty \in \R / \Z$ as $n\to \infty$ for every $i\notin J$. Hence $B_n(z)$ converges to a Blaschke-type holomorphic function $B_\infty(z)$ with fewer factors, and thus its restriction to $\s$ has less modality. In fact, if  $r_i^\infty=1$ for some  $i\notin J,$ then the corresponding factor $G_{i,n}(z):=((z-a_i^n)/(1-\overline {a}_i^nz))^{k_i}$ converges uniformly to the constant $(-a_i^\infty)^{k_i}$ on compact subsets of $\mathbb C \setminus \{a_i^\infty\}$. We can remove the singularities of $B_\infty$ at $a_i^\infty$ in this case. Thus, for $n$ sufficiently large, where $|a_i^n|>1$, $G_{i,n}$  contributes with at most two critical points on the circle and the corresponding critical points at $a_i^n$ and $1/\overline{a}_i^n$. We conclude that for $n$ sufficiently large, the modality of $B_n$ is at most $2m-2|J|<2m$, a contradiction.  \end{proof}

\begin{lem}\label{lem:collapse}
Fix $j\in\{1,\dots,m\}$ and let $\mu_n\in cc(\Delta), n\in \mathbb N,$ be a sequence of parameters satisfying  $$F_{\mu_n\kappa}(C_{2j-1}(\mu_n)) - F_{\mu_n\kappa}(C_{2j}(\mu_n)) \to k_j$$ as $n\to \infty$.
Then  
$r_j(\mu_n)\to 1$ and $C_{2j-1}(\mu_n)-C_{2j}(\mu_n)\to 0.$ 
Equivalently, the two turning points  
$
c_{2j-1}(\mu_n)=e^{2\pi i C_{2j-1}(\mu_n)}$ and $c_{2j}(\mu_n)=e^{2\pi i C_{2j}(\mu_n)}
$
satisfy  
$
{\rm dist}_{\s}(c_{2j-1}(\mu_n), c_{2j}(\mu_n)) \to 0$ as $n\to \infty$. 
\end{lem}

\begin{proof}
We first assume by contradiction that, up to a subsequence, $r_j(\mu_n)=|a_j(\mu_n)| \to r_j^\infty>1$ as $n \to \infty$, see Lemma \ref{lem:rj-bounded}. Following the notation and conclusions of Lemma \ref{lem:gap-kj}, we see that $\Delta \gamma_{j,2}(\mu_n) \to 0$ as $n\to \infty$. Since the argument of $B_{\mu_n \kappa}$ is strictly increasing along $\gamma_{2,j}(\mu_n)$, which encloses $a_j^n$, and $B_{\mu_n, \kappa}$ maps $\gamma_{2,j}(\mu_n)$ into $\s$, we conclude that up to a subsequence $B_{\mu_n \kappa}$ converges locally to a Blaschke-type map $B_\infty$, so that $B_\infty$ has infinitely many critical points in the annulus $1<|z|<r_j^\infty$. This follows from the fact that the arc $\gamma_{j,2}^n$ intersects this annulus for every $n$. Hence, $B_\infty$ is constant. This is a contradiction since the factor $((z-a_j^n)/(1-\overline {a}_j^nz))^{k_j}$ is not converging to a constant, and the remaining terms cannot cancel it out in a way that $B_\infty$ becomes a constant. This implies that $r_j(\mu_n)\to 1$ as $n\to \infty$. 

Hence, we may now assume that $a_j(\mu_n) \to a_j^\infty \in \s$ as $n\to \infty$, and that all the other parameters of $\mu_n$ converge. The factor $((z-a_j^n)/(1-\overline {a}_j^nz))^{k_j}$ then locally converges to a constant $(-a_j^\infty)^{k_j}$ on $\mathbb{C} \setminus \{a_j^\infty\}$. Thus, for $n$ sufficiently large, the modality of $f_{\mu_n \kappa}$ is at most $2m-2$ outside any given small neighborhood of $a_j^\infty \in \s$. Hence, for $n$ sufficiently large, both $c_{2j-1}(\mu_n)$ and $c_{2j}(\mu_n)$ must be contained in this neighborhood. This implies that ${\rm dist}_\s(c_{2j-1}(\mu_n),c_{2j}(\mu_n))\to 0$ as $n\to \infty$. In fact, using the argument in the previous paragraph, one can further show that the arc $\gamma_{2,j}(\mu_n)$ is arbitrarily close to $a_j^\infty$ for $n$ sufficiently large.
\end{proof}

The next lemma gives a uniform control on the first and second derivatives of the
lifts $F_{\mu\kappa}$ when the combinatorial type $\tau$ is fixed.
Geometrically, it shows that once the type $\tau$ is prescribed and the vector
$\kappa$ is chosen above $\tau$, no parameter $\mu\in\Delta_\tau$ can approach the
degenerate situation where some $a_j$ tends to the unit circle.
Indeed, by Lemma \ref{lem:collapse}, such a degeneration would force the corresponding
critical gap
$
F_\mu(C_{2j-1}(\mu)) - F_\mu(C_{2j}(\mu))
$
to approach its maximal possible value $k_j$, contradicting the combinatorial
constraints imposed by $\tau$.
The lemma below proves the existence of a uniform
$\delta>0$ such that $r_j(\mu)\ge 1+\delta$ for every $\mu\in\Delta_\tau$. As a consequence, the terms appearing in the expressions for $F'_\mu$ and $F''_\mu$ stay uniformly bounded.
This fact will play a key role in the construction of the Thurston pull-back map
in Section~6.

\begin{lem}\label{lem:deriv-bounds}
Let $g\in \mathcal{G}_m$ be a $2m$–modal circle map of type $\tau$, and
choose $\kappa=(k_0,\dots,k_m)$ so that 
$k_j\ge -\tau_{2j-1}+2$ for $j=1,\dots,m$, and 
$k_0=d+\sum_{j=1}^m k_j$, where $d$ is the degree of $g$.
Let $\Delta_\tau\subset cc(\Delta)$ be the set of parameters for which the
lift $F_{\mu\kappa}$ has type $\tau$. Then there exists $M_\tau>0$ such that
$$
|F_{\mu\kappa}'(t)| \leq M_\tau
\qquad\text{and}\qquad
|F_{\mu\kappa}''(t)|\leq M_\tau, \qquad \forall t\in \R, \forall \mu \in \Delta_\tau.
$$
Moreover, there exists $\delta>0$ so that $r_j(\mu) > 1+\delta$ for every $\mu \in \Delta_\tau$. 
\end{lem}

\begin{proof}
By the definition of $\Delta_\tau$ and the choice of $\kappa$, we have
\begin{equation}\label{eq:gap-tau}
0<
F_{\mu \kappa}(C_{2j-1}(\mu))
-
F_{\mu \kappa}(C_{2j}(\mu))
\leq 
-\tau_{2j-1} \leq
k_j-2,
\qquad j=1,\dots,m.
\end{equation}

Suppose by contradiction that there exist sequences $\mu_n=(\eta_{0}^{n}, a_{1}^{n}, r_{2}^{n}, \eta_{2}^{n}, \ldots, r_{m}^{n}, \eta_{m}^{n}) \in cc(\Delta)$ and $t_n \in [0,1]$, with $n\in \mathbb N$, such that $|F_{\mu_n \kappa}'(t_n)|\to +\infty$ as $n\to \infty$. We may assume that $\mu_n$ converges to $\mu_\infty=(\eta_0^\infty,a_1^\infty,\ldots, r_m^\infty, \eta_m^\infty)$ as $n\to \infty$, see Lemma \ref{lem:rj-bounded}.  we immediately see from the expression \eqref{efelinha} that there exists a non-empty set $J \subset \{1,\ldots,m\}$ so that $r_j^{n}\to 1$ as $n\to \infty$ for every $j\in J$. Otherwise, every $r_j^{n}=r_j(\mu_n)$ stays away from $1$, and thus the derivative is uniformly bounded. We may assume that $t_n \to t_*\in [0,1]$, $a_j^{n} \to a_*=e^{2\pi i t_*}\in \s$ for some non-empty subset $J_*\subset J$, and  $a_i^\infty\neq a_*$ for every $i\in J \setminus J_*$. For those terms in $J_*$, we see that in a small fixed neighborhood around $a_* \in \s$ and $n$ sufficiently large, the total variation 
$$
\sum_{j\in J_*} F_{\mu_n \kappa}(C_{2j-1}(\mu_n))
-
F_{\mu_n \kappa}(C_{2j}(\mu_n))
$$
is arbitrarily close to $\sum_{j\in J_*} k_j$. This implies that there exists $j_*\in J_*$ such that $F_{\mu_n \kappa}(C_{2j_*-1}(\mu_n))
-
F_{\mu_n \kappa}(C_{2j_*}(\mu_n))$ is arbitrarily close to $k_{j_*}$ for $n$ sufficiently large, a contradiction. We also conclude that there exists $\delta>0$ so that $r_j(\mu)>1+\delta$ for every $\mu \in \Delta_\tau$. In particular, the expression for $F_{\mu\kappa}'$ in \eqref{efelinha} implies that, since the $a_j$'s stay away from $\s$ uniformly in $\mu\in \Delta_\tau$, the derivatives $F_{\mu \kappa}'(t)$ and $F_{\mu \kappa}''(t)$ are uniformly bounded for every $t\in \R$ and every $\mu\in \Delta_\tau$. 
\end{proof}

The boundary $\partial (cc(\Delta))$, is composed by those
parameters $\mu$ for which at least two turning points collapse.
This is the only possible way to escape $\Delta$ since the turning
points are non-degenerate and the modality of $f_{\mu\kappa}$ is at most
$2m$. However, there are two different ways of how these turning points may be collapsed:

\begin{enumerate}
\item The parameter 
$\mu=(\eta_0,\, a_1,\, r_2,\, \eta_2, \ldots, r_m, \, \eta_m) \in      \partial\Delta$
is such that $r_j  > 1,$ for all $j \in \{1,\ldots,m\}$.
This happens when at least two turning points collapse to produce a
degenerate critical point of $F_{\mu\kappa}$. Their critical values are
also collapsed.
\item This second case  happens when
$r_j =1$, for some $j \in \{1, \ldots, m \}$. When this occurs,
two consecutive turning
points of $F_{\mu\kappa}$ are also collapsed, but their critical values
do not. This implies non-uniform bounds of $F_{\mu\kappa}'$ for
nearby parameters in $\Delta$. 
\end{enumerate}

By Lemma \ref{lem:deriv-bounds}, the boundary $\partial \Delta_\tau$, with $\kappa \succeq \tau,$ $\Delta_\tau \subset cc(\Delta)$, contains only parameters $\mu$ corresponding to degenerate critical points and those parameters $\mu$ corresponding to the change of type, i.e.,  $F_{\mu \kappa}(C_{2j-1}(\mu)) - F_{\mu \kappa}(C_{2j}(\mu))$ is a positive integer.

\section{Parameters and critical values}

Let us keep the above notation, that is,   $\Pi: \mathbb{R} \to \s$ is the 
covering map given by $\Pi(t) = e^{2 \pi i t}$ and 
$F_{\mu \kappa}$ is a lift of $f_{\mu \kappa}$.
The critical points of $F_{\mu\kappa}$, as in Proposition~\ref{prop:lifts}, are given by smooth functions globally defined in a connected component $cc(\Delta)\subset \Delta$, and labeled
as $C_1(\mu) < \cdots < C_{2m}(\mu) < C_1(\mu) +1,$ where 
$C_1(\mu) =\Pi(c_1(\mu))$
is a point of local maximum associated with $a_1$.
More generally, $C_{2j-1}(\mu)$ and $C_{2j}(\mu)$
are, respectively,  points of maximum and minimum of $F_{\mu \kappa}$, and
$0 < F_{\mu \kappa}(C_{2j-1}(\mu) )- F_{\mu \kappa}(C_{2j}(\mu) ) < k_j$, see Lemma \ref{lem:gap-kj}. 
 We choose the lift such that $F_{\mu \kappa}(C_1(\mu)) \in [0,1)$.

Given  $\kappa = (k_0, \ldots, k_m)$, recall that the topological degree
of $f_{\mu \kappa}$ is $d=k_0 - \sum_{j=1}^mk_j$. Let
$$
\begin{aligned}
     V& :=  \big\{\nu =(v_1,\ldots,v_{2m}) \in \mathbb{R}^{2m}:
(-1)^{i}(v_{i+1}-v_i)>0, \, 0<v_{2j-1}-v_{2j} < k_j \\ &    
\mbox{ and }
v_{2m} < v_{1}+d, \forall i=1,\ldots,2m, \forall j=1,\ldots,m \big\}.
\end{aligned}
$$

Consider the smooth  map  $\phi:cc(\Delta)  \to  V$  defined by
\begin{equation} \label{F}
\phi(\mu):=(F_{\mu}(C_1(\mu)),\ldots,F_{\mu}(C_{2m}(\mu))), \qquad \forall \mu \in cc(\Delta).
\end{equation}

Notice that Lemma \ref{lem:gap-kj} implies that $\phi$ is well-defined. The proposition below shows that $\phi$ is indeed a diffeomorphism from the space of parameters $cc(\Delta)$ onto the space of critical values $V$.

\begin{pro} \label{localdifeo} 
The map $\phi: cc(\Delta) \to V$  defined in \eqref{F} is a diffeomorphism.
\end{pro}

\begin{proof} We use the inverse function theorem 
to prove the statement. In $cc(\Delta)$, we consider real coordinates $\mu= (\eta_0, ,a_1, r_2, \eta_2, \ldots, r_m, \eta_m)$ as before and denote by $J(\mu)$ the Jacobian matrix of $\phi$ with respect to them. Since $F_\mu'(C_j(\mu))=0,$ for every $\mu$ and $j$, we see that $J(\mu)$ contains only terms of the form $\partial_{\mu_i} F_{\mu}(C_j(\mu))$.

Following \eqref{varfi}, we let $\varphi_j\in \R$ satisfy
$$
e^{2\pi i \varphi_j}=\frac{z-a_j}{|z-a_j|}, \qquad z=e^{2\pi i t} \in \s, \quad a_j=r_je^{2\pi i \eta_j}, \quad j=2,\ldots,m,
$$ 
and see it as a real-valued function depending on the real parameters $t,r_j,\eta_j$. We then compute
$$
2\frac{\partial \varphi_j}{\partial \eta_j}= - \frac{a_j}{ z -
a_j} -\frac{\overline {a}_j}{\overline{z} - \overline {a}_j}\quad \mbox{ and } \quad  4
\pi i \frac{\partial \varphi_j}{\partial r_j}(t)= \frac{1}{r_j} \left(
-\frac{a_j}{z - a_j} + \frac{\overline {a}_j}{\overline{z} - \overline {a}_j}\right).
$$
We also have that
$$
4 \pi i \frac{\partial \varphi_1}{\partial a_1}(t)=  -\frac{ 1}{ z -
a_1} + \frac{1}{\overline{z} - a_1} \ \  \mbox{ and } \ \ \ \frac{\partial
F_{\mu}}{\partial \eta_0}(t) = 1.
$$
Using \eqref{theta_til} and the above formulas, and performing  some algebraic manipulations, such as summing up columns and multiplying columns of $J$ by non-zero
constants, we end up with the following complex matrix
$$
J_2= \begin{pmatrix} 1&
\frac{a_1}{c_1-a_1} +\frac{\frac{1}{a_1}}{c_1-\frac{1}{a_1}}&
\frac{1}{c_1-a_2} & \frac{1}{c_1-\frac{1}{\overline {a}_2}} & \ldots &
\frac{1}{c_1-a_m} & \frac{1}{c_1-\frac{1}{\overline {a}_m}}\\
\vdots& \vdots& \vdots& \vdots& \vdots& \vdots& \vdots\\
1& \frac{a_1}{c_{2m}-a_1}
+\frac{\frac{1}{a_1}}{c_{2m}-\frac{1}{a_1}} & \frac{1}{c_{2m}-a_2} &
\frac{1}{c_{2m}-\frac{1}{\overline {a}_2}} & \ldots &
\frac{1}{c_{2m}-a_m} & \frac{1}{c_{2m}-\frac{1}{\overline {a}_m}} \\
\end{pmatrix}  
$$
where $c_i =c_i(\mu) = e^{2 \pi i C_i(\mu)}$. Notice that the determinant of $J$ does not vanish if and only if
the determinant of $J_2$ does not vanish. 
From (\ref{fl_a}) we know that
\begin{equation} \label{Jacobian_Mtil}
k_0+k_1\left(\frac{a_1}{c_i-a_1}-\frac{\frac{1}{a_1}}{c_i-\frac{1}{a_1}}
\right)+
\sum_{j=2}^{m}k_j\left(\frac{a_j}{c_i-a_j}-\frac{\frac{1}{\bar
a_j}}{c_i-\frac{1}{\overline {a}_j}}\right)=0\ , 
\end{equation}
for each $i=1,\ldots,2m$.

Now multiplying the first column of $\stackrel{\sim}{J}$ by $k_0$
and using (\ref{Jacobian_Mtil}) we find, after some algebraic
manipulations, the following Cauchy  matrix

$$
J_3= \begin{pmatrix}
\frac{1}{c_1-a_1}&
 \frac{1}{c_1-\frac{1}{a_1}}&
\frac{1}{c_1-a_2} & \frac{1}{c_1-\frac{1}{\overline {a}_2}} & \ldots &
\frac{1}{c_1-a_m} & \frac{1}{c_1-\frac{1}{\overline {a}_m}}\\
\vdots& \vdots& \vdots& \vdots& \vdots& \vdots& \vdots\\
\frac{1}{c_{2m}-a_1}&  \frac{1}{c_{2m}-\frac{1}{a_1}} &
\frac{1}{c_{2m}-a_2} & \frac{1}{c_{2m}-\frac{1}{\overline {a}_2}} & \ldots
&
\frac{1}{c_{2m}-a_m} & \frac{1}{c_{2m}-\frac{1}{\overline {a}_m}} \\
\end{pmatrix}
$$
Again, the determinant of $J_3$ does not vanish if
and only if the determinant of $J$ does not vanish.
But the determinant of $J_3$ is given by

\begin{equation}
 \frac{\prod_{ 1\leq i< j \leq 2m } (c_j-c_i)(b_j-b_i)} {\prod _{
1\leq i, j \leq 2m } (c_j-b_i)}
\end{equation}
where $b_1=a_1, \, b_2 = \frac{1}{a_1}, \, b_3=a_2, \, b_4 = \frac{1}{\bar
a_2},\ldots, b_{2m-1} = a_m, \, b_{2m}=\frac{1}{\overline {a}_m}$. We conclude
that $\det J \neq 0$ and this implies that $\phi$ is a local diffeomorphism.

Denote by $\phi_i(\mu) = F_{\mu}(C_i(\mu)), i=1,\ldots, 2m$,
the components of $\phi(\mu)$. Lemma \ref{lem:gap-kj} implies  that $0 < \phi_{2j-1}(\mu) -\phi_{2j}(\mu)  < k_j$
which means that $\phi$  maps $\Delta$ into $V$.  Since $\phi$ is a local diffeomorphism, the image of $\phi$ is open in $V$. It is also closed since $\partial(cc(\Delta))$ contains only parameters for which either some consecutive critical values collapse or $\phi_{2j-1}(\mu)-\phi_{2j}=k_j$ for some $j$. Hence $\phi$ is surjective. Since $V$ is simply connected, we conclude that $\phi$ is a diffeomorphism. \end{proof}

Let $g \in \mathcal{G}_m$  be a $2m$-multimodal map of degree $d$ 
and  $\tau =( \tau_1, \ldots, \tau_{2m-1})$ its  type with respect to some 
point $c_1$ of local maximum.
We choose $\kappa =(k_0, \ldots ,k_m)$, which will be fixed, so that  
$k_j \geq  -\tau_{2j-1} + 2$, $j = 1, \ldots, m$, and $k_0 = d + \sum_{j=1}^{m}k_j$. We remark here that, according to Lemma~\ref{lem:deriv-bounds}, we need to choose the components of $\kappa$  adding at least $2$ to the corresponding components of $\tau$ in order to realize all possible distances between critical values of two consecutive critical points. This choice is also important to guarantee lifts $F_{\mu \kappa}$ with derivative bounds as stated in the lemma below. Denote by $\Delta_{\tau}$ the subset of $\Delta$ corresponding to those parameters $\mu$ for which
 $F_{\mu\kappa}$ has type $\tau$. We also denote by $V_{\tau}\subset V$
the image of $\Delta_{\tau}$ under $\phi$ and
$\phi_{\tau}=\phi|_{\Delta_{\tau}}.$

\begin{lem}\label{simplyconnect} If  $\tau$ and $\kappa$  are as above,
then there is a constant $M_{\tau} > 0  $ such that
$\|F^{\prime}_{\mu}\| < M_{\tau}$ and $\|F^{\prime
\prime}_{\mu}\| < M_{\tau }$ for any $\mu \in \Delta_{\tau}$.
Moreover,  the map $\phi_{\tau}$ is a diffeomorphism between 
each connected component of $\Delta_\tau$ and $V_\tau$,
in particular $\Delta_\tau$  is simply connected.
 
\end{lem}

\begin{proof}
It follows from Lemma~\ref{lem:deriv-bounds} that a bound for the derivatives 
$F_{\mu}^{\prime}$ and $F_\mu^{\prime \prime}$ fail only 
when $\phi_{2j-1}(\mu) - \phi_{2j}(\mu)$ is  near $k_j$, for some
$j \in \{1, \ldots, m\}$. By our choice of $k_j$,
$0 < \phi_{2j-1}(\mu) - \phi_{2j}(\mu) \leq - \tau_{2j-1} \leq k_{j} -2$.
Therefore, there is a constant $M_\tau$ as stated.

The boundary 
$\partial \Delta_\tau$ consists of parameters $\mu \in \Delta$, where either two consecutive critical points coincide (and the corresponding critical values also coincide) or  $\phi_{2j-1}(\mu) - \phi_{2j}(\mu) = k_j,$ for some
$j \in \{1, \ldots, m\}$. Then $\phi_\tau$ can be continuously extended to 
a map from the
closure of $cc(\Delta_\tau) $ to the closure of $V_\tau$ and  $\phi_\tau^{-1}(\nu) \cap \Delta_\tau$ cannot accumulate in 
$\partial \Delta_\tau$.
This implies that $\phi_\tau$ is proper and by 
Proposition~\ref{localdifeo} it is a local diffeomorphism,
therefore, $ \phi_\tau: \Delta_\tau \to V_\tau$ is a covering map. 
Since $V_\tau$ is connected and  simply connected, $\phi$ is
a diffeomorphism when restricted to a connected component of
$\Delta_\tau$ and each connected component of $\Delta_\tau$
is simply connected. \end{proof}

\section{Realizing finite combinatorics - the Thurston operator} \label{finite}

Let $g: \s \to \s $ be a $C^1$ $2m$-multimodal map in $\mathcal{G}_m$,  i.e., $g$ does not have wandering 
intervals, inessential periodic attractors, or non-trivial intervals of periodic points 
with the same period. We
consider a lift $G: \R \to \R$ of $g$ with respect to the covering map
$\Pi : \R \to \s $ given by $\Pi(t) = e^{2 \pi i t}$. 
We assume with no loss of generality  that $0$ is a local maximum of
$G$ and $0 \leq G(0) < 1$. Let $ 0 = C_1 < C_2 < \cdots < C_{2m} < 1$ 
be the $2m$ turning points of $G$, let $\tau = (\tau_1, \ldots, \tau_{2m-1} ) $ 
be the type of $g$ 
with respect to $ C_1 $ and let  $\kappa=(k_0, \ldots, k_m)$ be formed by positive integers satisfying \eqref{cond_kj}, i.e., 
where $k_0 = d + \sum_{j=1}^{m}k_j$ and $k_j \geq -\tau_{2j-1} +2,$ for
$j=1,\ldots, m$.  

With the above $\kappa$ fixed, we consider  the corresponding 
family of $2m$-multimodal maps  $f_\mu = f_{\mu\kappa},$ where
$\mu$  is in a connected component of $ \Delta_\tau$. 
According to Corollary \ref{label}, the $2m$ critical points 
of $f_\mu$  are well defined and cyclically ordered in $\s$. They are labeled
according to the cyclic order on $\s$, say $c_1(\mu) < \cdots < c_{2m}(\mu)$, 
where $c_1(\mu) $ is a local maximum, which by definition is
related to the parameter $a_1$ in the sense that 
$c_1(\mu),c_2(\mu) \in \s \cap \Gamma_1$. 
 Let  $F_{\mu }= F_{\mu \kappa}$  
 be a  lift of $f_{\mu }$ with respect to the above
covering map $\Pi: \R \to \s$. 
The turning points of $F_\mu$ are 
$C_1(\mu) < \cdots < C_{2m}(\mu) < C_1(\mu) +1$
and they satisfy  $\Pi(C_i(\mu)) = c_i(\mu)$, $i=1, \ldots, m$.
As a consequence of our choice of $\kappa$,
if $ \mu $ is in a connected component of $\Delta_\tau$, then
$F_{\mu \kappa}(C_{2j-1}(\mu) )- F_{\mu \kappa}(C_{2j}(\mu) ) 
\leq  - \tau_{2j-1} \leq  k_j-1$,
for $j=1, \ldots, m$ and Lemma~\ref{simplyconnect} implies that
$\|F^{\prime}_{\mu}\| < M_{\tau}$ and $\|F^{\prime
\prime}_{\mu}\| < M_{\tau }$.

Now we consider a normalization of the family $F_\mu$, $\mu \in \Delta_\tau$, 
given by the following family  
$$
\tilde F_{\mu}(t):= F_{\mu}(t+C_1(\mu))-C_1(\mu), \quad \mu \in \Delta_\tau .
$$
Observe that $t=0$ is a local maximum of $\tilde F_{\mu}$ and corresponds 
to the local maximum $C_1(\mu)$ of $F_{\mu}$. Moreover, both  
$F_{\mu}$ and $\tilde F_{\mu}$ induce in $\mathbb{S}^1$ the same map, 
after a conjugation by a rigid rotation.

The family $\tilde F_{\mu}$, $\mu \in \Delta_\tau$, has essentially the same 
properties of the family $F_{\mu}$. Indeed,
the critical points of $\tilde F_\mu$ are 
$ 0= \tilde C_1(\mu)  < \tilde C_2(\mu) < \cdots < \tilde C_{2m}(\mu) < 1$ 
and they satisfy  $\tilde C_j(\mu)=C_j(\mu)-C_1(\mu)$. 
The corresponding critical values are
$ \tilde \phi_j(\mu)=\tilde F_{\mu}(\tilde C_j(\mu)) = F_{\mu}(C_j(\mu))-C_1(\mu)$.
Therefore, the  map
 $\tilde \phi: \Delta_\tau \to V_\tau$, defined by
 $ \tilde \phi(\mu)=(\tilde \phi_1(\mu),\ldots,\tilde \phi_{2m}(\mu))$
satisfies $ \tilde \phi(\mu) = \phi(\mu) - C_1(\mu)(1,1,\ldots,1)$ and 
it is easy to check that if $D \tilde \phi(\mu)$ 
is the Jacobian matrix of $\tilde \phi$ at $\mu \in \Delta_\tau$, then
$D \tilde \phi(\mu)=D \phi(\mu)-M(\mu)$,
where $D \phi(\mu)$ is the Jacobian matrix of $\phi$ at $\mu \in \Delta_\tau$ and 
$M(\mu)$ is a matrix with constant entries in each column $j \in\{1,\ldots, 2m\}$ 
given by $\partial C_1(\mu)/\partial \mu_j$. Since the first column of 
$D\phi(\mu)$ has constant entries equal to $1$, we have 
$\det(D\tilde \phi(\mu)) =\det(D \phi(\mu))\neq 0$ for all $\mu\in\Delta_0$.
It follows that $\tilde \phi$ is a local diffeomorphism and inherits all the properties proved for $\phi$. Arguing as before, we conclude that
Lemma~\ref{simplyconnect} is true if we replace $F_\mu$ by $\tilde F_\mu$
and $\phi$ by $\tilde \phi$.
 
For simplicity, from now on, we will keep using the notation 
$\phi$, $F_{\mu}$ and $f_\mu$
instead of $\tilde \phi$, $\tilde F_{\mu}$ and $\tilde f_\mu$, 
having in mind that the 
family  $F_{\mu}$ is normalized as above.

Now we fix a connected component $\textrm{cc}(\Delta_\tau)$ of 
$\Delta_\tau$  as above and
assume that $g$    has \emph{ finite combinatorics} which 
means that the union of the forward orbits of its
turning points, the so-called \emph{ post-critical set} of $g$,  is a finite set. We claim that there is $\mu \in \textrm{cc}(\Delta_\tau)$
such that $g$ is topologically conjugate to $f_{\mu}$. 
To prove this claim e consider the following data:

\begin{enumerate}
\item
The points $z_1 < \ldots < z_k \in [0,1)$ are the turning points of
$G $ union with  all their iterates $\mz$ and the points $z_i$ such that $ G(z_i)
\in \Z$. Observe that $z_1 = C_1 = 0$.
\item
Let $t_1, \ldots,t_{2m+l}$ be the positive integers such that
$z_{t_1},\ldots,z_{t_{2m}}$ are the turning points of $G$
and for $j=2m+1, \ldots, 2m+l$, the points $z_{t_j}$ are not turning
points and $G(z_{t_j})\in \Z$.
\item
Let $\sigma:\{1, \ldots, k \} \to \{1, \ldots, k \}$ be the map given by
$z_{\sigma(j)} = G(z_j) \mz$ if $ j  \not \in \{t_{2m+1},\ldots,t_{2m+l}
\}$  and $\sigma(j)=1 $ for $j \in  \{t_{2m+1},\ldots,t_{2m+l} \}$.
\end{enumerate}

For each $r\in \{1,\ldots,k-1\}$, we define $s(r)\in \N$ to be the least integer $l$ such that $g^l([z_r,z_{r+1}])$ contains a turning point. Because $g\in \mathcal G_m,$ $s(r)$ is well-defined and $s(r)=0$ if and only if $z_r$ or $z_{r+1}$ is a turning point. Indeed, if no such integer exists, then $g$ admits a periodic interval $[z_i,z_{i+1}]$ with no turning points, which forces $g$ to have either an inessential periodic attractor or a non-trivial interval of periodic points, contradicting the fact that $g\in \mathcal G_m$. See Proposition \ref{prop_combinatorio}.
 
With this information about $g$, the choice of $\kappa = (k_0, \ldots, k_m)$
and  the family $f_{\mu}$, $\mu \in \textrm{cc}(\Delta_\tau)$ (and its
lift  $F_{\mu}$) already defined above, we define a continuous map
$T = T_{\tau}$ on the following simplex:
$$
W= \left\{x = (x_1, \ldots, x_k) \in  \R^k : 0 = x_1 < \ldots x_k < 1 \right\}.
$$
In order to  define $T(x_1, \ldots, x_k) = (y_1, \ldots, y_k)$,  
remember that $0 = C_1(\mu) < \cdots < C_{2m}(\mu)$ are the critical points
of $F_\mu$ and take  the unique $\mu \in  \textrm{cc}(\Delta_{\tau})$ such that:
$$
\phi(\mu) =\left(x_{\pi(t_1)} , x_{\pi(t_2)}, \ldots,  x_{\pi(t_{2m})} \right) +
\left(0,\tau_1,\tau_1+\tau_2, \ldots, \sum_{j=1}^{2m-1} \tau_j \right).
$$
Now we define $(y_1, \ldots, y_k) := T(x_1, \ldots, x_k)$ as follows:
\begin{enumerate}
\item
$0= y_{t_1} < \cdots <  y_{t_{2m}} < 1$ 
are the turning points of $F_{\mu}$ in $[0,1)$.
\item
$y_{t_{2m+1}}, \ldots, y_{t_{2m+l}}$ are the points such that
$F_{\mu}(y_{t_j}) \in \Z$.
\item
For $t_j < i < t_{j+1}$, $y_i$ is the unique point in $(y_{t_j}, y_{t_{j+1}})$
so that $x_{\sigma(i)}= F_{\mu}(y_i) \mz$.
\end{enumerate}

\medskip

The map $T$ is a single-valued map from $W$ to itself, which is
called the {\em Thurston operator of type $\tau$} associated to the family
$f_{\mu}$,  $\mu \in \textrm{cc}( \Delta_{\tau})$.

\begin{pro}
    $T:W \to W$ is continuous.
\end{pro}

\begin{proof}
Notice that the construction of $T$
depends on two steps: first, the choice of the parameter $\mu \in \Delta_\tau$
corresponding to a given configuration of critical values, and second, the
pullback of marked points through the lift $F_\mu$ along prescribed monotonicity
branches. By Proposition \ref{localdifeo}, the map assigning $\mu$ to a configuration is continuous. Moreover, the family $F_\mu$ depends smoothly on
$\mu$ and, on each monotonicity interval, the inverse branches vary continuously.
Since the combinatorial data fixes the choice of branches and integer lifts, each coordinate of $T(x)$ is obtained by solving an equation of the form
$
F_\mu(y) = x_j + n,
$
which has a unique solution on the prescribed branch and depends continuously on
$x$. It follows that $T$ is continuous.
\end{proof}

Observe that a fixed point
$(x_1, \ldots,  x_k)$ of $T$ means that $ x_{\pi(i)} = F_{\mu}(x_i) \mz$, 
for some $\mu \in \textrm{cc}(\Delta_{\tau})$. This implies
that $f_{\mu }$ and $g$ are combinatorially equivalent and, because
$g \in \mathcal{G}_m$, they are topologically conjugate, see \cite{MS}. 

The following combinatorial property will be used in the proof of the existence of a fixed point of the Thurston operator.

\begin{pro}\label{prop_combinatorio}
Let $g:\mathbb S^{1}\to \mathbb S^{1}$ be a $C^{1}$ multimodal map in the class $G_{m}$, and let
$
z_{1}<z_{2}<\cdots<z_{k}<z_{1}+1
$
be the control points used in the definition of the Thurston operator.
For each $j\in\{1,\dots,k\}$, define $s(j)\in\mathbb Z_{\ge 0}$ to be the smallest
nonnegative integer such that
\[
g^{\,s(j)}([z_j,z_{j+1}])
\]
contains a turning point of $g$. Then $s(j)$ is well-defined for every $j$. Moreover, if $s(j)\ge 1$, then
$g([z_j,z_{j+1}])$
contains an interval of the form $[z_r,z_{r+1}]$ for some $r\in\{1,\dots,k\}$ satisfying
$
s(r)=s(j)-1.
$
\end{pro}

\begin{proof}
We first show that $s(j)$ is well-defined. Fix $j\in\{1,\dots,k-1\}$ and suppose,
by contradiction, that no iterate of $[z_j,z_{j+1}]$ contains a turning point. Then,
for every $n\ge 0$, the interval
$
I_n:=g^n([z_j,z_{j+1}])
$
contains no turning point, and therefore $g$ is monotone on each $I_n$.

Since the endpoints of $[z_j,z_{j+1}]$ belong to the finite set of control points,
their forward orbits are finite. Hence, only finitely many endpoint configurations can
occur and, therefore, there exist integers $n_1<n_2$ such that
$
I_{n_1}=I_{n_2}.
$
Thus $I:=I_{n_1}$ is a periodic interval for some iterate $g^p$, where $p=n_2-n_1$,
and by construction $I$ contains no turning point.

Now consider the restriction $g^p|_I$. Since $I$ contains no turning point, this map
is monotone. If $g^p|_I$ or $g^{2p}|_I$ is the identity on $I$, then $I$ is a nontrivial interval
consisting entirely of periodic points of the same period, contradicting the definition
of $G_m$. Otherwise, $g^p|_I$ is a nontrivial monotone self-map of $I$, hence $I$
contains an attracting periodic orbit for $g$. Because $I$ contains no turning point,
the immediate basin of this attractor contains no turning point, so the attractor is
inessential, again contradicting the definition of $G_m$. This proves that $s(j)$ is
well-defined.

Let us prove the second statement. Assume that $s(j)\ge 1$. Then
$
J:=g([z_j,z_{j+1}])
$
is an interval, and by minimality of $s(j)$, the interval
$
g^{\,s(j)-1}([z_j,z_{j+1}])
$
contains no turning point, while
$
g^{\,s(j)}([z_j,z_{j+1}])=g(J)
$
does contain one.  Since $J$ is a non-trivial connected interval, it must contain at least one 
sub-interval $[z_r,z_{r+1}]$ so that
$
g^{\,s(j)-1}([z_r,z_{r+1}]) \subset g^{\,s(j)}([z_j,z_{j+1}])
$
contains a turning point. On the other hand, by minimality of $s(j)$,
no earlier iterate of $[z_j,z_{j+1}]$ contains a turning point, and therefore
no earlier iterate of $[z_r,z_{r+1}]$ can contain one either. It follows that
$
s(r)=s(j)-1,
$
as claimed.
\end{proof}

\section{Existence of a fixed point for the Thurston operator}

In this section, we present an elementary proof of the existence of a fixed point
for Thurston's pull-back operator in the Blaschke-type setting.  
The argument makes use of Brouwer's fixed–\-point theorem and differs from the case of interval maps in an essential
aspect: the dynamics lives on the circle, and critical arcs may wrap around
$\mathbb S^1$. Thus, the operator depends on the {\em type}
$\tau$ of the map, i.e., the combinatorial data recording the relative position
of critical points and their images.  In particular, the pull–back equations
involve integer winding numbers, absent in the interval case.
Also, the proof presented here differs substantially from the original argument of  de~Melo and van Strien~\cite{MS}, where the existence of a fixed point for the Thurston operator is obtained by a more elaborate topological argument relying on detailed control of the boundary of parameter space.  In contrast, our construction exploits the uniform $C^{2}$ bounds available in the Blaschke-type family and the combinatorial covering property (Proposition \ref{prop_combinatorio}), which
together allow us to define a ``thickened'' simplex $W_\varepsilon$ that is forward-invariant under the Thurston operator $T$.
Since $W_\varepsilon$ is a compact convex polytope with nonempty
interior, the existence of a fixed point follows immediately from the classical
Brouwer Fixed Point Theorem.  This provides a simpler and more transparent route to existence than the general argument in~\cite{MS}, while remaining fully compatible with the additional winding data that arise in the
circle setting.

Throughout, we assume the hypotheses of Lemma~\ref{simplyconnect}
(and of the combinatorial lemma that produces the integers $s(i)$). Recall that the Thurston operator
$T = T_\tau : W \to W$
is defined on the open simplex
\[
W := \Big\{x = (x_1,\dots,x_k)\in\mathbb R^k : 0 = x_1 < x_2 < \cdots < x_k < 1\Big\}.
\]
Given $x\in W$, we first find the unique parameter
$\mu = \mu(x)\in\Delta_\tau$ such that the vector of critical values
of $F_\mu$ coincides with the prescribed coordinates of $x$, with the
appropriate winding corrections determined by the type $\tau$. We then
define $y=T(x)$ by pulling back the remaining coordinates of $x$
along $F_\mu$, as in the previous subsection. Thus $T(x)$ encodes the
configuration of “control points’’ for the map $F_\mu$ realizing the
critical values prescribed by $x$.

Note that $W$ is open, so we cannot apply Brouwer’s fixed-point theorem
directly. Following a standard idea, we shall construct a nested family
of closed sub-simplices on which $T$ is invariant, and then apply
Brouwer to one of these closed sets.

\medskip

Let $c>0$ be the uniform bound for the first and second derivatives of
the lifts $F_\mu$ obtained in Lemma~\ref{simplyconnect}, and let
$s(i)\in\mathbb N$, $i=1,\dots,k-1$, be the integers provided by Proposition \ref{prop_combinatorio} which depend only on the dynamics of $g$ and control the number of iterates needed for the images of the intervals
\([z_i,z_{i+1}]\) to contain a turning point.

\begin{defi}
For each sufficiently small $\varepsilon > 0$, we define the closed simplex
\[
W_\varepsilon
:= \Big\{x=(x_1,\dots,x_k)\in W :
x_{i+1}-x_i \;\ge\; \frac{\varepsilon}{c^{s(i)}} \ \text{ for all }
i=1,\dots,k-1\Big\}.
\]
\end{defi}

Geometrically, $W_\varepsilon$ is obtained from $W$ by forbidding the
collapse of two consecutive coordinates $x_i,x_{i+1}$ faster than the
scale $\varepsilon/c^{s(i)}$. This scale is chosen so that the
distortion estimates given by Lemma~\ref{simplyconnect} and the
combinatorics encoded in the numbers $s(i)$ fit together.

\begin{lem}\label{lem:Weps-basic}
For every sufficiently small $\varepsilon>0$, the set $W_\varepsilon$
is a non–empty, compact, convex subset of $\mathbb R^k$ with non–empty
interior. In particular, $W_\varepsilon$ is homeomorphic to a closed
ball of dimension $k-1$.
\end{lem}

\begin{proof}
The defining inequalities
$x_{i+1}-x_i \ge \varepsilon/c^{s(i)}, i=1,\dots,k-1,$
together with $0=x_1$ and $x_k<1$, describe the intersection of $W$
with finitely many closed half–spaces. Hence $W_\varepsilon$ is closed,
convex, and bounded. Hence, it is compact. For $\varepsilon>0$ small enough, we can choose a point $x\in W$ with gaps $x_{i+1}-x_i$ all much larger
than $\varepsilon/c^{s(i)}$, so $x\in W_\varepsilon$ and the interior
of $W_\varepsilon$ is non–empty. Since $W$ is an open simplex of
dimension $k-1$, each $W_\varepsilon$ is an affine closed simplex of the
same dimension, hence homeomorphic to a closed $(k-1)$–ball.
\end{proof}

The next lemma is the key “invariance’’ property.

\begin{lem}\label{lem:Weps-invariance}
There exists $\varepsilon_0>0$ such that for every
$0<\varepsilon\le \varepsilon_0$ we have
$
T(W_\varepsilon)\subset W_\varepsilon.
$
\end{lem}

\begin{proof}
We argue by contradiction. Suppose that there exists a sequence
$\varepsilon_n\to 0^+$ and points $x^n\in W_{\varepsilon_n}$ such that
$y^n := T(x^n)\notin W_{\varepsilon_n}$ for every $n$.
Write
$x^n = (x^n_1,\dots,x^n_k)$ and $y^n = (y^n_1,\dots,y^n_k).$
By the definition of $W_{\varepsilon_n}$ and of the Thurston operator,
each $x^n$ determines a map $F_{\mu_n}$, $\mu_n\in\Delta_\tau$,
whose critical values (with the appropriate type corrections) coincide
with certain coordinates of $x^n$, and where the remaining
coordinates of $y^n$ are obtained by pulling back along $F_{\mu_n}$.

Since $y^n\notin W_{\varepsilon_n}$, there exists an index
$m=m(n)\in\{1,\dots,k-1\}$ such that
\begin{equation}\label{eq:gap-x}
|y^n_{m+1}-y^n_m|
< \frac{\varepsilon_n}{c^{s(m)}}.
\end{equation}
Passing to a subsequence, we may assume that $m(n)\equiv m$ is constant.

Consider now the interval
\[
I^{n}_0 := [y^n_m, y^{n}_{m+1}] \subset [0,1).
\]

Suppose first that $s(m)=0\Rightarrow |y^n_{m+1}-y^n_m| <\varepsilon_n$. Then $y^n_m$ or $y^n_{m+1}$ is a turning point of $F_{\mu_n}$. Without loss of generality, we assume that $y^n_m$ is a turning point. Then $F_{\mu_n}'(y^n_m)=0.$ Let $r\in \{1,\ldots,k\}$ be such that the interval $[x^n_r,x^n_{r+1}]$, shifted by some integer, is contained in the interval determined by $F_{\mu_n}(y^n_m)$ and $F_{\mu_n}(y^n_{m+1})$. We may assume that $r$ does not depend on $n$. Since $x\in W_{\varepsilon_n}$, we have
$$
\begin{aligned}
0<\frac{\varepsilon_n}{c^{s(r)}} & \leq |x^n_r-x^n_{r+1}|\leq |F_{\mu_n}(y^n_m)-F_{\mu_n}(y^n_{m+1})| = |F_{\mu_n}'(\zeta_n)||y^n_m-y^n_{m+1}|\\
& \leq |F_{\mu_n}''(\hat \zeta_n)||\zeta_n - y^n_{m}|\varepsilon_n\leq c_2 \varepsilon_n | \zeta_n- y^n_m|, 
\end{aligned}
$$
where $\zeta_n\in (y^n_m,y^n_{m+1})$, $\hat \zeta_n \in (y^n_m,\zeta_n)$, and $c_2 :=\sup _{\mu_n\in \Delta_\tau, x\in \R}|F_{\mu_n}''(x)|>0$. We conclude that 
$$|\zeta_n-y^n_m| \geq \frac{1}{c_2c^{s(r)}}, \quad \forall n.$$ This contradicts the fact that $|\zeta_n - y^n_m| \leq |y^n_m-y^n_{m+1}|<\varepsilon_n\to 0$ as $n\to \infty$.

Now assume that $s(m) \geq 1$. From the construction of the Thurston operator, the image of $I^n_0$ under $F_{\mu_n}$ contains an interval of the form $[x^n_r,x^n_{r+1}]+q$ for some $r\in \{1,\ldots,k\}$ and $q\in \Z$, satisfying $s(r)=s(m)-1$, see Proposition \ref{prop_combinatorio}. Then  
$$
\begin{aligned}
    c\frac{\varepsilon_n}{c^{s(m)}}> c|y^n_m-y^n_{m+1}|\geq |x^n_r-x^n_{r+1}|\geq \frac{\varepsilon_n}{c^{s(r)}}=\frac{\varepsilon_n}{c^{s(m)-1}},
\end{aligned}
$$
a contradiction. This contradiction shows that for all sufficiently small $\varepsilon>0$,
$T(W_\varepsilon)\subset W_\varepsilon$.
\end{proof}

We can now complete the fixed–point argument.

\begin{theo}\label{thm:thurston-fixed-point}
Let $g\in\mathcal G_m$ be a post–critically finite $2m$–multimodal map
of type $\tau$, and let $T: W\to W$ be the associated Thurston
operator constructed above. Then $T$ admits a fixed point
$x^*\in W$.
\end{theo}

\begin{proof}
Choose $0<\varepsilon\le\varepsilon_0$ as in Lemma~\ref{lem:Weps-invariance}.
By Lemma~\ref{lem:Weps-basic}, the set $W_\varepsilon$ is a compact,
convex subset of $\mathbb R^k$ with non–empty interior, hence
homeomorphic to a closed ball of dimension $k-1$. By
Lemma~\ref{lem:Weps-invariance}, we have
$T(W_\varepsilon)\subset W_\varepsilon.
$
Since $T$ is continuous, Brouwer’s fixed point theorem applied to the
map $T|_{W_\varepsilon}: W_\varepsilon\to W_\varepsilon$ yields a
point $x^*\in W_\varepsilon$ such that $T(x^*)=x^*$. In particular,
$x^*\in W$ is a fixed point of the Thurston operator $T_\tau$.
\end{proof}

\begin{remark}
This proof of the existence of a fixed point for the Thurston operator is considerably simpler than the classical argument, which involves iterating the operator and analyzing its asymptotic behavior near the boundary. The key point is that, once we have uniform bounds on the derivatives of the family
$f_\mu$ and a combinatorial control encoded in the integers $s(i)$, we can work directly on a suitable closed simplex $W_\varepsilon$ and
invoke the standard Brouwer fixed-point theorem. 

It is well known that the Thurston operator is a contraction, and thus its iterates, starting from any initial condition, converge to the unique fixed point. This is illustrated in the example below
\end{remark}

\begin{figure}[ht]
\centering
\includegraphics[width=0.5\textwidth]{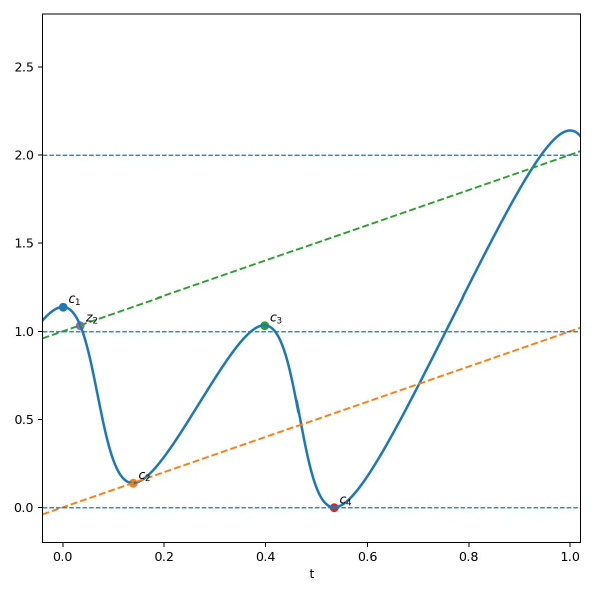}
\caption{Normalized lift associated with the fixed point of the Thurston operator in Example \ref{example_1}.}
\label{fig:thurston-fixed-point-lift}
\end{figure}

\begin{example}\label{example_1}
\emph{ We consider the combinatorics of a $4$-modal map ($m=2$), with $d=1$, and $5$ control points in the post-critical trajectories $0=z_1<z_2<z_3<z_4<z_5<1$, whose forward orbits are determined by
\[
\sigma(1)=3,\qquad \sigma(2)=2,\qquad \sigma(3)=3,\qquad \sigma(4)=2,\qquad \sigma(5)=1,
\]
with type data
\[
\tau_1=-1,\qquad \tau_2=0,\qquad \tau_3=-1.
\]
The critical points are labeled  $c_1=z_1,$ $c_2=z_3,$ $c_3=z_4$ and $c_4=z_5$, while $z_2$ corresponds to a fixed point in the forward orbit of $c_3$.
We work in the Blaschke family
\[
B_\mu(z)=e^{2\pi i\eta_0}z^7
\left(\frac{z-a_1}{1-a_1z}\right)^3
\left(\frac{z-a_2}{1-\overline a_2 z}\right)^3,
\qquad a_2=r_2e^{2\pi i\eta_2},
\]
and iterate the Thurston operator on the control points. At each step, the parameter $\mu$ is obtained by solving the corresponding normalized critical-value equations, with the lift normalized by
\[
\widetilde F(t)=F_\mu(t+C_1(\mu))-C_1(\mu),
\]
where $C_1(\mu)$ is the first local maximum. The iteration converges to a fixed configuration
\[
(z_1,z_2,z_3,z_4,z_5)\approx(0, 0.03427,\,0.13811,\,0.39748,\,0.53431).
\]
The corresponding Blaschke parameters are approximately
\[
\eta_0\equiv 0.69690 \pmod 1,\qquad
a_1\approx 1.31911,\qquad
r_2\approx 1.33310,\qquad
\eta_2\approx 0.60207,
\]
so that
$
a_2\approx 1.33310\,e^{2\pi i\,0.60207}.
$
In the corresponding normalized lift, the critical points satisfy
$
0= c_1<c_2< c_3< c_4<1,
$
with
$
\widetilde F(c_1)=z_3+1,
\widetilde F(c_2)=z_3,
\widetilde F(c_3)=z_2+1,
$ and $\widetilde F(c_4)=0.$
In particular, $z_2$ is a fixed point modulo $\mathbb{Z}$, see Figure \ref{fig:thurston-fixed-point-lift}.
}
\end{example}

\section{Proof of main theorem}

\subsection{Existence}

Let us start proving that $2m$-multimodal maps $f_{\mu \kappa}$ on the circle induced by Blaschke-type maps $B_{\mu \kappa}$ do not admit pathological intervals and thus are in $\mathcal G_m$. 
\begin{pro}\label{prop:fmukappa in G_m}
    Let $f_{\mu \kappa}=B_{\mu \kappa}|_{\mathbb S^1}, \mu \in \Delta.$ Then $f_{\mu \kappa} \in \mathcal G_m$.
\end{pro}

\begin{proof} We have to prove that $f_{\mu \kappa}$ satisfies properties (i)-(iv) in Definition \ref{defi_2m_modal}. It is immediate that $f_{\mu \kappa}$ admits no non-trivial interval of periodic points of the same period since $f_{\mu \kappa}$ is real-analytic.  Hence, (i) holds. That $f_{\mu \kappa}$ has no wandering intervals follows from the fact that any $C^2$ endomorphism of the circle with only non-flat critical points admits no wandering intervals, see \cite[Theorem A, Chapter IV]{MS}. This implies (ii). That $f_{\mu \kappa}$ has no inessential attractors follows from the general fact that the immediate basin of attraction of an attracting (or one-sided attracting) periodic point of a rational map contains at least one critical point, see \cite[Theorems 8.6 and 10.15]{Mil06}. Since $\mu \in \Delta$, any such critical point $c\in \overline{\mathbb{C}}$ of $B_{\mu \kappa}$ necessarily lies in $\mathbb S^1$, because the other critical points of $B_{\mu \kappa}$ are mapped to the fixed points $0,\infty$. Hence, $c$ must be a turning point of $f_{\mu \kappa}$. This proves (iii). Property (iv) trivially holds. \end{proof}

Let $g\in \mathcal G_m$ be a $2m$-multimodal map satisfying conditions (i)-(iv) as in Definition \ref{defi_2m_modal}. Let $d$ be the degree of $g$ as an endomorphism of $\mathbb S^1$ and let $\tau = (\tau_1,\ldots,\tau_{2m-1})$ be its type. Let $\kappa = (k_0, k_1,\ldots,k_m)$ be an $m$-tuple of positive integers satisfying \eqref{cond_kj}.  Denote by $T:W \to W$ the Thurston map associated with the type $\tau$ and the family $f_{\mu \kappa} = B_{\mu \kappa}|_{\mathbb S^1}$, $\mu\in \Delta_\tau \subset \Delta,$ of $2m$-modal maps of the circle with type $\tau$. By Theorem \ref{thm:thurston-fixed-point}, $T$ has a fixed point corresponding to a parameter $\mu_*\in \Delta_\tau$ so that  $f_{\mu_*\kappa}$ is combinatorially equivalent to $g$. In other words, denoting the turning points of $g$ and $f_{\mu_*\kappa}$ by $c_1, \ldots, c_{2m}\in \mathbb S^1$ and ${c}_1', \ldots, {c}_{2m}'\in \mathbb S^1$, respectively, there exists a bijection between the post-critical sets
\[
h: \bigcup_{i=1}^{2m} \bigcup_{n\geq 0}\{g^n(c_i)\} \to  \bigcup_{i=1}^{2m} \bigcup_{n\geq 0}\{f_{\mu_*\kappa}^n( c'_i)\},
\]
which preserves the cyclic order and satisfies $h(g^n(c_i))=f_{\mu_* \kappa}^n( c'_i)$ for every $i,n.$ Since $g$ and $f_{\mu_*\kappa}$ are combinatorially equivalent,  $h$ naturally extends to the union of all pre-images of the turning points of $g$ so that $h$ is still order preserving and conjugates $g$ and $f_{\mu_*\kappa}$ on that set. If a turning point of $g$ is periodic, then it is attracting, and $h$ extends to the basin of attraction of this attracting periodic orbit using fundamental domains both for $g$ and $f_{\mu_* \kappa}$. Since $g$ and $f_{\mu_*\kappa}$ satisfy properties (i)-(iii), $h$ is already an order preserving bijection conjugating $g$ and $f_{\mu_*\kappa}$ defined on a dense subset of $\mathbb S^1$ with dense image. Hence, $h$ continuously extends to an orientation-preserving homeomorphism $h:\mathbb S^1 \to \mathbb S^1$ conjugating $g$ and $f_{\mu_*\kappa}$. This proves the existence part of Theorem \ref{Theo_Principal}.

\subsection{Uniqueness}
We now turn to the question of {\it
uniqueness\/}. We shall prove that uniqueness holds in the case of
post-critically finite maps within the family $f_{\mu} = f_{\mu \kappa}: \mathbb S^1 \to \mathbb S^1$, with $\kappa$ fixed.
The main tool to be used in the proof of this result is the
criterion due to Thurston given by Theorem~\ref{thurston}. In order to use this criterion in our setup, we need the
description of the pre-image of the unit circle under a Blaschke-type
product $B_{\mu\kappa}$ given by Lemma~\ref{toplemma2}.

\begin{theo} Let $B_1, B_2$ be two post-critically finite
Blaschke-type products from the family $B_{\mu}=B_{\mu \kappa}$, $\mu \in \Delta$ and $\kappa$ fixed.
If the restrictions to the unit circle of these maps are topologically conjugate $2m$-multimodal maps, then, up to conjugation by a
rotation of the complex plane, $B_1$ and $B_2$ are the same.
\end{theo}

\begin{proof}
Note that, as observed in Section 2, the orbifolds of $B_1$ and $B_2$ are hyperbolic. Therefore, by Thurston's criterion, all that has to be done is to show that $B_1$ and $B_2$ are combinatorially equivalent (as post-critically finite branched coverings of the Riemann sphere).

In keeping with the notation of Lemma \ref{toplemma2}, we write
$c_j'(B_1),c_j''(B_1)$ for the critical points and $D_j(B_1)$ for the
disks appearing in that lemma, emphasizing the dependence of such
objects on $B_1$, and similarly for $B_2$. Note that the post-critical
sets $P_{B_1}$ and $P_{B_2}$ are both contained in
$\mathbb S^1\cup\{0,\infty\}$.

Let $h:\mathbb S^1\to \mathbb S^1$ be the topological
conjugacy between $B_1\big{|}_{\mathbb S^1}$ and
$B_2\big{|}_{\mathbb S^1}$, which we assume
orientation-preserving. For the sake of what follows, there is no
loss of generality in assuming also that $h(c_j'(B_1))=c_j'(B_2)$ and
$h(c_j''(B_1))=c_j''(B_2)$, for all $j$. Let
$H_0:\overline{\mathbb{C}} \to \overline{\mathbb{C}}$ be an
orientation-preserving homeomorphism such that
$H_0\big{|}_{\partial \mathbb{D}}\equiv h$ and $H_0(0)=0,
H_0(\infty) = \infty$. Such an $H_0$ is easily constructed, with the
additional property of being symmetric with respect to inversion
about the unit circle

As we saw in the proof of Lemma \ref{toplemma2}, for each
$j=1, 2, \ldots, m,$ the restricted maps
$$
B_1:D_j(B_1)\cap \mathbb{D}\to \overline{\mathbb{C}}\setminus \mathbb{D}
\quad\text{and}\quad B_2:D_j(B_2)\cap \mathbb{D} \to
\overline{\mathbb{C}}\setminus \mathbb{D}
$$
are both $k_j$-to-$1$ branched covering maps, branched at
$1/ \overline{a}_j(B_1)$ and $1 / \overline{a}_j(B_2)$, respectively.
Note that all domains here are topological disks. Hence there exists
an orientation-preserving homeomorphism $\phi_j:D_j(B_1)\cap
\mathbb{D} \to D_j(B_2) \cap \mathbb{D}$ such that the diagram
\[
\begin{CD}
D_j(B_1)\cap \mathbb{D}@>{\phi_j}>>D_j(B_2) \cap \mathbb{D}\\ @V{B_1}VV             @VV{B_2}V\\ \overline{\mathbb{C}}\setminus \mathbb{D}@>>{H_0}> \overline{\mathbb{C}} \setminus \mathbb{D}
\end{CD}
\]
commutes, and such that
$\phi_j\big{|}_{D_j(B_1) \cap\partial\mathbb{D}} \equiv H_0\equiv h$.
The remaining part of the boundary of the domain of $\phi_j$ is
mapped by $B_1$ into the unit circle. Thus, in any case, we see that
for each $z$ in the boundary of the domain of $\phi_j$ we have
\begin{equation} \label{conjone}
h(B_1(z)) \;= B_2(\phi_j(z)) \ .
\end{equation}

Now let $V_{B_1}$ be the topological disk $\mathbb{D}\setminus
\bigcup_{j=1}^{m} \overline{D_j(B_1)}$, and let $V_{B_2}$ be similarly
defined. These disks are mapped (by $B_1$ and $B_2$ respectively) onto
the unit disk, and the restricted maps $B_1:V_{B_1}\to \mathbb{D}$ and
$B_2:V_{B_2}\to \mathbb{D}$ are $k_0$ to $1$ branched covering maps
branched at the origin. Hence, there exists an orientation-preserving
homeomorphism $\phi:V_{B_1} \to V_{B_2}$ such that the diagram
\[
\begin{CD}
V_{B_1}@>{\phi}>>V_{B_2}\\
@V{B_1}VV             @VV{B_2}V\\
\mathbb{D}@>>{H_0}>\mathbb{D}
\end{CD}
\]
commutes, and such that $\phi\big{|}_{\partial V_{B_1}\cap
\partial\mathbb{D}}\equiv H_0\equiv h$. As before, the remaining
part of the boundary of $V_{B_1}$ is mapped into the unit circle.
Therefore we have, for all $z\in V_{B_1}$,
\begin{equation} \label{conjtwo}
h(B_1(z))\;=\;B_2(\phi(z)) \ .
\end{equation}

Next, let us consider the union of all these homeomorphisms $\phi_j$
and $\phi$. Because \eqref{conjone} and \eqref{conjtwo} hold true,
these homeomorphisms weld together across their common boundaries
and yield an orientation-preserving homeomorphism
$\psi:\mathbb{D}\to \mathbb{D}$. Finally, extending $\psi$ by
reflection (geometric inversion) across $\partial\mathbb{D}$, we get
an orientation-preserving homeomorphism $H_1$ of the Riemann sphere.
This $H_1$ is homotopic to $H_0$ relative to $\partial\mathbb{D}\cup
\{0,\infty\}$ by construction, and moreover makes the following
diagram commute
\[
\begin{CD}
\overline{\mathbb{C}}@>{H_1}>>\overline{\mathbb{C}}\\
@V{B_1}VV             @VV{B_2}V\\
\overline{\mathbb{C}}@>>{H_0}>\overline{\mathbb{C}}
\end{CD}
\]
This proves that $B_1$ and $B_2$ are Thurston equivalent. By
Theorem \ref{thurston}, $B_1$ and $B_2$ are conformally conjugate. A
conformal conjugacy between $B_1$ and $B_2$ must fix $0$ and $\infty$
and must leave the unit circle invariant. Hence, it must be a
rotation.
\end{proof}

\hfill\newline
\noindent{\bf Acknowledgments.} EdF and EV are partially supported by the São Paulo Research Foundation (FAPESP), Brazil, Grant No. 2023/07076-4. EV is partially supported by FAPESP Grant No. 2025/08532-9. EV thanks Mitsuhiro Shishikura and Hiroyuki Inou for fruitful conversations.   PS is partially supported by the National Natural Science Foundation of China, Grant No. w2431007. EdF, EV and PS thank the support of the Shenzhen International Center for Mathematics - SUSTech. 
 
\printbibliography[title={References}]
%\bibliography{blaschke}

\end{document}